\documentclass[10pt]{amsart}
\usepackage{amssymb, amsmath}

\newtheorem{theorem}{Theorem}
\newtheorem{lemma}[theorem]{Lemma}
\newtheorem{proposition}[theorem]{Proposition}
\newtheorem{corollary}[theorem]{Corollary}
\newtheorem{definition}[theorem]{Definition}

\newcommand{\newsection}[1] {\section{#1}\setcounter{theorem}{0}
 \setcounter{equation}{0}\par\noindent}

\newcommand{\R}{{\mathbb R}}
\newcommand{\C}{{\mathbb C}}
\newcommand{\Z}{{\mathbb Z}}

\newcommand{\g}{{\rm g}}

\newcommand{\dstar}{\textit{div}\,}

\renewcommand{\l}{\lambda}

\newcommand{\T}{{\mathcal T}}
\newcommand{\I}{{\mathcal I}}

\begin{document}

\title[Spectral cluster bounds]
{Sharp $L^p$ bounds on spectral clusters\\
  for Lipschitz metrics} \thanks{HK was supported by the DFG through
  SFB 611. HS was supported by NSF grant DMS-0654415. DT was supported
  by NSF grant DMS-0801261 and by the Miller Institute.}
\author{Herbert Koch} \address{Mathematisches Institute, University of
  Bonn, Beringstrasse 1, 53115 Bonn, Germany}
\email{koch@math.uni-bonn.de} \author{Hart F. Smith}
\address{Department of Mathematics, University of Washington, Seattle,
  WA 98195} \email{hart@math.washington.edu} \author{Daniel Tataru}
\address{Department of Mathematics, University of California,
  Berkeley, CA 94720} \email{tataru@math.berkeley.edu}

\begin{abstract}
  We establish $L^p$ bounds on $L^2$ normalized spectral clusters for
  self-adjoint elliptic Dirichlet forms with Lipschitz
  coefficients. In two dimensions we obtain best possible bounds for
  all $2\le p\le \infty$, up to logarithmic losses for $6<p\le 8$. In
  higher dimensions we obtain best possible bounds for a limited range
  of $p$.
\end{abstract}

\maketitle

\newsection{Introduction} Let $M$ be a compact, 2-dimensional manifold
without boundary, on which we fix a smooth volume form $dx$. Let $\g$
be a section of positive definite symmetric quadratic forms on
$T^*(M)$, and let $\rho$ be a strictly positive function on $M$.

Consider the eigenfunction problem
\begin{equation*}
  \dstar(\g\,d \phi)+\l^2\rho \phi=0\,,
\end{equation*}
where $\dstar$, which maps vector fields to functions, is the dual of
$d$ under $dx$.  This setup includes as a special case eigenfunctions
of the Laplace-Beltrami operator.  We refer to the real parameter $\l$
as the frequency of $\phi$, and take $\l\ge 0$.  Under the condition
that $\g$ and $\rho$ are measurable, with strictly positive lower and
upper bounds, there exists a complete orthonormal basis
$\{\phi_j\}_{j=1}^\infty$ of eigenfunctions for $L^2(M,\rho\, dx)$,
with frequencies satisfying $\l_j\rightarrow\infty$.

In this paper we prove the following, where for convenience of the
statement we take $\l\ge 2$.  Except for the factor $(\log \l)^\sigma$
in \eqref{qboundd}, these bounds are the best possible for general
Lipschitz $\g$ and $\rho$, in terms of the growth in $\l$, by an
observation of Grieser \cite{Gr}.

\begin{theorem}\label{theorem1.1}
  Suppose that $\g,\rho\in {\rm Lip}(M)$.  Assume that the frequencies
  of $u$ are contained in the interval $[\l,\l+1]$, so that
  \begin{equation*}
    u=\sum_{j:\l_j\in[\l,\l+1]} c_j\,\phi_j\,.
  \end{equation*}
  Then
  \begin{equation}\label{qboundc}
    \|u\|_{L^{p}(M)}\le C_p\, \lambda^{\frac 12-\frac 2p}\|u\|_{L^{2}(M)}\,,\quad 8<p\le\infty\,,
  \end{equation}
  and
  \begin{equation}\label{qboundd}
    \|u\|_{L^{p}(M)}\le C \lambda^{\frac 23(\frac 12-\frac 1p)}(\log\l)^\sigma
    \|u\|_{L^{2}(M)}\,,\quad 6 \le p\le 8\,,
  \end{equation}
  where $\sigma=\frac 32$ for $p=8$, and $\sigma=0$ for $p=6$.
\end{theorem}

For $2\le p\le 6$, the bound \eqref{qboundd} is already known from
\cite{Sm}, without logarithmic loss,
\begin{equation*}
  \|u\|_{L^{p}(M)}\le C \lambda^{\frac 23(\frac 12-\frac 1p)}
  \|u\|_{L^{2}(M)}\,,\quad 2\le p\le 6\,.
\end{equation*}

To put the above estimates in context, we recall the previously known
results.  In case $\g$ and $\rho$ are $C^\infty$, Sogge \cite{So}
established the following bounds in general dimensions $d\ge 2$,
\begin{equation}\label{sogge}
  \begin{split}
    \|u\|_{L^p(M)}&\le C\,\l^{d(\frac 12-\frac 1p)-\frac
      12}\,\|f\|_{L^2(M)}\,, \qquad p_d\le p\le\infty\,,
    \\
    \|u\|_{L^p(M)}&\le C\,\l^{\frac{d-1}2(\frac 12-\frac
      1p)}\,\|f\|_{L^2(M)}\,, \qquad 2\le p\le p_d\,,\qquad
    p_d=\tfrac{2(d+1)}{d-1}\,,\rule{0pt}{15pt}
  \end{split}
\end{equation}
which are best possible at all $p$ for unit width spectral clusters.
Semiclassical generalizations were obtained by Koch-Tataru-Zworski
\cite{KTZ}.

The estimates \eqref{sogge} were extended to $C^{1,1}$ coefficients in
\cite{Sm1}.  On the other hand, examples constructed by Smith-Sogge
\cite{SmSo} show that for small $p$ they can fail for coefficients of
lesser H\"older regularity. In particular, for Lipschitz coefficients,
the following would in general be best possible
\begin{equation}\label{lipconj}
  \begin{split}
    \|u\|_{L^p(M)}&\le C\,\l^{d(\frac 12-\frac 1p)-\frac
      12}\,\|f\|_{L^2(M)}\,, \qquad \tfrac{2(d+2)}{d-1}\le
    p\le\infty\,,
    \\
    \|u\|_{L^p(M)}&\le C\,\l^{\frac{2(d-1)}{3}(\frac 12-\frac
      1p)}\,\|f\|_{L^2(M)}\,, \qquad 2\le p\le
    \tfrac{2(d+2)}{d-1}\,.\rule{0pt}{15pt}
  \end{split}
\end{equation}

The second estimate in \eqref{lipconj} was established in \cite{Sm} on
the range $2\le p\le p_d$, as well as the first for $p=\infty$.  That
proof proceeded by establishing the no-loss estimate \eqref{sogge} on
sets of diameter $\l^{-\frac 13}$, the scale on which the Lipschitz
coefficients can be suitably approximated by $C^2$ coefficients. This
scaling had been used in \cite{Tat} to prove Strichartz estimates with
loss for wave equations with Lipschitz coefficients, and examples
showing optimality of those estimates were constructed in
\cite{SmTa}. The same idea occurs in this paper in the $\l^{-\frac
  13}$ time scale expansion of $u$ in terms of simple tube solutions.
For metrics of H\"older regularity $C^s$ with $s<1$, the optimal
bounds, and corresponding examples, were obtained in \cite{KST1} for
$2\le p\le p_d$, as well as $p=\infty$. For $s<1$ there can occur
exponentially localized eigenfunctions, and as a result the $p=\infty$
bounds are strictly worse than in the case of Lipschitz coefficients.

Establishing optimal bounds for $p_d<p<\infty$ is significantly more
involved, since in this case it is no longer sufficient to prove
uniform bounds on $u$ over small sets. One needs in addition to bound
the possible energy overlap between such sets; that is, to consider
energy flow for rough equations on a scale where the coefficients are
not well approximated by $C^2$ functions.  The first advance in this
direction was made in \cite{KST2}, where for $d=2$ bounds were
obtained which, while not optimal, did improve upon those obtained by
interpolating the optimal bounds for $p=p_d$ and $p=\infty$. In a
related direction, the bounds \eqref{lipconj} with $d=2$ were
established for all $p$ by Smith-Sogge \cite{SmSo2} for smooth
Dirichlet forms on two-dimensional manifolds with boundary, with
either Dirichlet or Neumann boundary conditions, a setting where the
exponents of \eqref{lipconj} are the best generally possible by
\cite{Gr}.  Such a manifold is treated in \cite{SmSo2} as a special
case of a form with Lipschitz coefficients on a manifold without
boundary, by extending the coefficients evenly across the boundary in
geodesic normal coordinates.

The new method of this paper is to combine energy flow estimates for
Lipschitz coefficients with combinatorial arguments to carry out, in
essence, a worst case scenario analysis.  Except for the factors of
$\log\l$ for $6<p\le 8$, this suffices to obtain the sharp results in
two dimensions. To the best of our knowledge, this is the first
work where the sharp intermediate endpoint between $p_d$ and infinity
has been reached for any problem of this type.

For dimensions $d\ge 3$ our methods yield partial
results.  Precisely, we also prove in this paper the first estimate in
\eqref{lipconj} for $\frac{6d-2}{d-1}< p\le\infty$, leaving the case
of $p_d<p\le \frac{6d-2}{d-1}$ still open.

For most of this paper we focus for simplicity on the case of two
dimensions, where our results are strongest. We start in section
\ref{sec:argument} with the reduction of Theorem \ref{theorem1.1} to
two key propositions, one involving short time dispersive estimates
and the other long time energy overlap bounds. The short time
estimates, Proposition \ref{p:shorttime}, are established in sections
\ref{sec:shorttime} and \ref{sec:bilinear}, through a combination of
Strichartz and bilinear estimates. The energy overlap bounds,
Proposition \ref{p:bushcount}, are established in sections
\ref{sec:bushcount} and \ref{sec:wavepacket}, and depend on energy
propagation estimates for Lipschitz metrics. In section
\ref{sec:highdimension} we establish the estimates of \eqref{lipconj}
in dimensions $d\ge 3$ for $p$ in the aforementioned subset of the
conjectured range.  The limitation on $p$ there is due in part to both
the low order of localization in the energy propagation estimates and
to our use of only Strichartz estimates for $d\ge 3$.


\newsection{The argument}\label{sec:argument} In this section, we use
paradifferential arguments and a frame of ``tube solutions" to reduce
the proof of Theorem \ref{theorem1.1} to the two key results of this
paper, Propositions \ref{p:bushcount} and \ref{p:shorttime}.  Since
the estimate \eqref{qboundd} follows by interpolation from the case
$p=8$ and the known case $p=6$, we consider $8\le p \le \infty$.  To
avoid unnecessary repetition, we focus on the case $p=8$ in the
following steps.  At the end of this section we show how to deduce the
estimate \eqref{qboundc} for $p>8$ by a simple interpolation argument
in the last step of the proof for $p=8$.

The spectral localization of $u$, integration by parts, elliptic
regularity, and the equation, yield the following over $M$
\begin{equation}\label{ellipticest}
  \l^{-1}\|\dstar(\g \, du)+\l^2\rho u\|_{L^2}+\l^{-1}\|du\|_{L^2}+\l^{-2}\|d^2\! u\|_{L^2}
  +\l\|u\|_{H^{-1}}
  \lesssim \|u\|_{L^2}\,.
\end{equation}
By choosing a partition of unity subordinate to suitable local
coordinates, for $p=8$ we are then reduced to the following.

\begin{theorem}\label{l8theorem}
  Suppose that $\g$ and $\rho$ are globally defined on $\R^2$, with
$$
\|\g^{ij}-\delta^{ij}\|_{Lip(\R^2)}+\|\rho-1\|_{Lip(\R^2)}\le c_0 \ll
1\,.
$$
Then the following estimate holds for functions $u$ supported in the
unit cube of $\R^2$,
\begin{equation}
  \|u\|_{L^8(\R^2)} \lesssim \lambda^\frac14 (\log\l)^{\frac 32}\bigl(\|u\|_{L^2(\R^2)} + 
  \lambda^{-1}\| \dstar (\g \,du)+ \lambda^2\rho u\|_{L^2(\R^2)}\bigr)\,.
  \label{l8}\end{equation}
\end{theorem}


{\bf Step 1:} {\em Reduction to a frequency localized first order
  problem.}  In proving Theorem \ref{l8theorem} we may replace the
function $\g$ by $\g_\l$, where $\g_\l$ is obtained by smoothly
truncating $\hat \g(\xi)$ to $|\xi|\le c\l$, $c$ some fixed small
constant.  Since
$$
\|\g-\g_\l\|_{L^\infty(\R^2)}\lesssim \l^{-1}\,,\qquad
\|\nabla(\g-\g_\l)\|_{L^\infty(\R^2)} \lesssim 1\,,
$$
the right hand side of \eqref{l8} is comparable to the same quantity
after this replacement.  Similarly we replace $\rho$ by $\rho_\l$.

By the Coifman-Meyer commutator theorem \cite{CM} (see also
\cite[Prop. 3.6.B]{Tay}), the commutator of $\g_\l$ or $\rho_\l$ with
a multiplier $\Gamma(D)$ of type $S^0$ maps $H^s\rightarrow H^{s+1}$
for $-1\le s\le 0$, so we may take a conic partition of unity to
reduce matters to establishing \eqref{l8} with $u$ replaced by
$\Gamma(D)u$, with $\Gamma(\xi)$ supported where $|\xi_2|\le
c\,\xi_1$. This step loses compact support of $u$, but we may still
take the $L^8$ norm over the unit cube. Finally, arguments as in
\cite[Corollary 5]{Sm} reduce matters to considering $\hat u(\xi)$
supported where $|\xi|\approx\l$.

We now label $x_1=t$, and $x_2=x$, and let $(\tau,\xi)$ be the dual
variables to $(t,x)$.  Thus, with $c$ above small, $\hat u(\tau,\xi)$
is supported where $\{|\xi|\le \frac 12\l\,,\tau\approx\l\}$.  For
$|\xi|\le \tfrac 34\l$ we can factor
$$
-\g_\l(t,x)\cdot(\tau,\xi)^2+\l^2\rho_\l(t,x)=-\g^{00}_\l(t,x)\bigl(\tau+\tilde
a(t,x,\xi,\l)\bigr) \bigl(\tau-a(t,x,\xi,\l)\bigr)\,,
$$
where $\tilde a\,,a>0$, and both belong to $\l C^1S_{\l,\l}$ (on the
interval of $\xi$ where they are defined) according to the following
definition.

\begin{definition}
  Let $b(t,x,\xi,\l)$ be a family of symbols in $(x,\xi)$ depending on
  parameters $t$ and $\l$.  We say $b(t,x,\xi,\l)\in
  S_{\l,\l^{\delta}}$ if, for all multi-indices $\alpha,\beta$,
$$
|\partial_{t,x}^\alpha\partial_\xi^\beta b(t,x,\xi,\l)|\le
C_{\alpha,\beta}\l^{-|\beta|+\delta|\alpha|}\,,
$$
where the constants $C_{\alpha,\beta}$ depend only on $\alpha$ and
$\beta$.

We say $b(t,x,\xi,\l)\in C^1S_{\l,\l^\delta}$
if 
the stronger estimate holds
$$
|\partial_{t,x}^\alpha\partial_\xi^\beta b(t,x,\xi,\l)|\le
C_{\alpha,\beta}\l^{-|\beta|+\delta\max(0,|\alpha|-1)}\,.
$$
We write $b\in \l^m S_{\l,\l^{\delta}}$ to indicate $\l^{-m}b\in
S_{\l,\l^\delta}$.
\end{definition}

In our applications either $\delta=1$ or $\delta=\frac 23$, and we
suppress the explicit $\l$ when writing $b$.  The symbols $b\in
S_{\l,\l}$ are simply a bounded family of $S^0_{0,0}(\R^2,\R)$ symbols
rescaled by $(t,x,\xi)\rightarrow (\l t,\l x,\l^{-1}\xi)$, so that
$L^2(\R)$ boundedness of $b(t,x,D)$, as well as the Weyl quantization
$b^w(t,x,D)$, follows (with uniform bounds in $t$ and $\l$) by
\cite{weyl}.  For symbols in $C^1S_{\l,\l^\delta}$, the asymptotic
laws for composition and adjoint hold to first order. In particular,
if $a\in C^1S_{\l,\l^\delta}$ then
$$
a^w(t,x,D)= a(t,x,D)+r(t,x,D)\,,\quad r\in \l^{-1}S_{\l,\l^\delta}\,,
$$
and if $a\in S_{\l,\l^\delta}$, $b\in C^1S_{\l,\l^\delta}$, then
$$
a(t,x,D)b(t,x,D)=(ab)(t,x,D)+r(t,x,D)\,,\quad r\in
\l^{-1}S_{\l,\l^\delta}\,.
$$

Let $a_\l(t,x,\xi)$ be obtained by smoothly truncating the
$(t,x)$-Fourier transform of $a(\cdot\,,\cdot\,,\xi,\l)$ to
frequencies less than $c\l$. Then since $a\in \l C^1 S_{\l,\l}$,
$$
a(t,x,\xi,\l)-a_\l(t,x,\xi)\in S_{\l,\l}\,.
$$
The symbol $a_\l(t,x,\xi)$ inherits from $a(t,x,\xi,\l)$ the
estimates, for $|\xi|\le\frac 34\l$,
\begin{equation}\label{acond1}
  a_\l(t,x,\xi)\approx \l\,,\qquad \partial_\xi^2 a_\l(t,x,\xi)\approx -\l^{-1}\,.
\end{equation}
We now extend $a_\l$ globally in $\xi$ so that $a_\l(t,x,\xi)=\l$ for
$|\xi|\ge(\frac 34-c)\l$, and $a_\l$ is unchanged for $|\xi|\le \frac
58 \l$, while maintaining the frequency localization in the $(t,x)$
variables.  In particular, \eqref{acond1} holds for $|\xi|\le\frac
58\l$.  A similar observation holds for $\tilde a$.

By the above, we have the factorization over $|\xi|\le \frac 58\l$,
$$
\dstar\g_\l\,d+\l^2\rho_\l= \g^{00}_\l\bigl( D_t+\tilde
a^w_\l(t,x,D)\bigr)\bigl(D_t-a^w_\l(t,x,D)\bigr)+ r(t,x,D)\,,
$$
where $r\in \l S_{\l,\l}$. Since $a^w_\l(t,x,D)u$ is supported where
$\tau\approx\l$, and $D_t+a_\l(t,x,D)$ admits a parametrix in
$\l^{-1}S_{\l,\l}$ there, we have thus reduced Theorem \ref{l8theorem}
to establishing the following estimate over $(t,x)\in [0,1]\times\R$,
\begin{equation*}
  \|u\|_{L^8([0,1]\times\R)} \lesssim 
  \lambda^\frac14 (\log\l)^{\frac 32}\bigl(\,\|u\|_{L^2([0,1]\times\R)} + 
  \| (D_t - a_\l^w(t,x,D))u\|_{L^2([0,1]\times\R)}\bigr)\,.
\end{equation*}

We denote by $S(t,s)$ the evolution operators for $a^w_\l(t,x,D)$,
which are unitary on $L^2(\R)$.  Precisely, $u(t,x)=S(t,t_0)f$
satisfies
\begin{equation}
  \bigl(D_t - a_\l^w(t,x,D)\bigr) u = 0 \,, \qquad u(t_0,\cdot\,) = f\,.
  \label{directional-hom}
\end{equation}
If $\hat f$ is supported in $|\xi|\le \frac 34\l$, then so is $\hat
u(t,\cdot\,)$, since $a_\l=\l$ for $|\xi|\ge (\frac 34-c)\l$, and
$a_\l$ is spectrally localized in $x$ to the $c\l$ ball.  By the
Duhamel formula, it then suffices to prove that
\begin{equation*}
  \|u\|_{L^8([0,1]\times\R)} \lesssim \lambda^\frac14 (\log\l)^{\frac 32}\|u_0\|_{L^2(\R)}\,,\qquad
  u=S(t,0)u_0\,,
\end{equation*}
with $\widehat u_0$ supported in $|\xi|\le \frac 34\l$, and
$a_\l(t,x,\xi)$ as above.

Henceforth, we will take $\|u_0\|_{L^2(\R)}=1$.


{\bf Step 2:} {\em Decomposition in a wave packet frame on $\l^{-\frac
    13}$ time slices.}  Let $a_{\l^{2/3}}(t,x,\xi)$ be obtained by
smoothly truncating the $(t,x)$-Fourier transform of
$a_\l(\cdot\,,\cdot\,,\xi)$, or equivalently that of $a(t,x,\xi,\l)$,
to frequencies less than $c\l^{\frac 23}$. Then $a_{\l^{2/3}}\in \l
S_{\l,\l^{2/3}}$, and $a_{\l^{2/3}}$ also satisfies \eqref{acond1},
for $|\xi|\le \frac 58\l$ in case of the second estimate in
\eqref{acond1}.

We divide the time interval $[0,1]$ into subintervals of length
$\l^{-\frac13}$, and thus write $[0,1]\times \R$ as a union of slabs
$[l\l^{\frac 13},(l+1)\l^\frac13]\times\R$.  Within each such slab we
will consider an expansion of $u$ in terms of homogeneous solutions
for $D_t-a^w_{\l^{2/3}}(t,x,D)$, We refer to the homogeneous solutions
on each $\l^{-\frac 13}$ time interval as \textit{tube solutions},
since they will be highly localized to a collection of tubes $T$.

The collection of tubes is indexed by triples of integers $T=(l,m,n)$,
with $0\le l\le \l^{\frac 13}$ referencing the slab $[l\l^{\frac
  13},(l+1)\l^\frac13]\times\R$.  We will describe the construction
for the slab $[0,\l^{-\frac 13}]\times\R$; the tube solutions
supported on the other slabs are obtained in an identical manner.
Thus, $T$ is here identified with a pair $(m,n)\in\Z^2$.

We start with a $\l^{\frac 23}$-scaled Gabor frame on $\R$, with
compact frequency support. That is, we select a Schwartz function
$\phi$, with $\hat\phi$ supported in $|\xi|\le \frac 98$, such that
for all $\xi$
$$
\sum_{n\in\Z}|\hat\phi(\xi-2n)|^2=1\,.
$$
It follows that, with $x_T=\l^{-\frac 23}m$ and $\xi_T=2\l^{\frac
  23}n$, the space-frequency translates
$$
\phi_T(x)=\l^{\frac 13} e^{ix\xi_T}\phi(\l^{\frac 23}(x-x_T))
$$
form a tight frame, in that for all $f\in L^2(\R)$,
$$
f=\sum_{m,n} c_T\phi_T\,,\quad\text{where}\quad c_T=\int
\overline{\phi_T(x)}f(x)\,dx\,,
$$
from which it follows that
$$
\|f\|^2_{L^2(\R)}=\sum_T |c_T|^2\,.
$$
Since the function $f$ in our application will be frequency restricted
to $|\xi|\le\l$, the index $\xi_T$ will only run over $|n|\le
\l^{\frac 13}$, by the compact support condition on $\hat\phi$.

The frame is not orthogonal, so it is not necessarily true, for
arbitrary coefficients $b_T$, that $\|\sum b_T
\phi_T\|_{L^2(\R)}^2\approx \sum |b_T|^2$. However, since the
functions $\widehat\phi_T$ have almost disjoint support for different
$\xi_T$, this does hold if the sum is over a collection of $T$ for
which the corresponding $\xi_T$ are distinct,
\begin{equation}\label{Gaborsum}
  \Bigl\|\sum_{T\in\Lambda} b_T \phi_T \Bigr\|_{L^2(\R)}^2\approx \sum_{T\in\Lambda} |b_T|^2 
  \quad \text{if}\quad
  \xi_T\ne\xi_{T'} \quad\text{when}\quad T\ne T'\in\Lambda\,.
\end{equation}

Let $v_T$ denote the solution to
$$
\bigl(D_t - a_{\l^{2/3}}^w(t,x,D)\bigr) v_T = 0 \,, \qquad
v_T(0,\cdot\,) = \phi_T\,.
$$
We define $x_T(t)$ by
$$
x_T(t)=x_T-t\,\partial_\xi a_{\l^{2/3}}(0,x_T,\xi_T)\,.
$$
For $t\in[0,\l^{-\frac 13}]$ the function $v_T(t,\cdot)$ is a
$\l^{\frac 23}$-scaled Schwartz function with frequency center
$\xi_T$, and spatial center $x_T(t)$, where the envelope function
satisfies uniform Schwartz bounds over $t$.  This follows, for
example, by Theorem \ref{tubeloc} or \cite[Proposition 4.3]{KT}.
Thus, $v_T$ is localized, to infinite order in $x$, to the following
tube, which we also refer to as $T$,
$$
T=\{(t,x):|x-x_T(t)|\le\l^{-\frac 23},\;t\in [0,\l^{-\frac 13}]\}\,.
$$

Since the $a^w_{\l^{2/3}}(t,x,D)$ flow is unitary, for each $t$ the
functions $\{v_T(t,\cdot\,)\}$ form a tight frame on $L^2(\R)$. On
each $\l^{-\frac 13}$ time slab we can thus expand
$$
u(t,x)=\sum_T c_T(t)\,v_T(t,x)\,, \qquad c_T(t)=\int
\overline{v_T(t,x)}\,u(t,x)\,dx\,.
$$
Differentiating the equation, we see that
$$
c_T'(t)=i\int
\overline{v_T(t,x)}\,\bigl(a_\l^w(t,x,D)-a^w_{\l^{2/3}}(t,x,D)\bigr)u(t,x)\,
dx\,.
$$
Since $a_\l(t,x,\xi)-a_{\l^{2/3}}(t,x,\xi)\in \l^{\frac 13}S_{\l,\l}$,
and $\|u_0\|_{L^2}=1$, we then have uniformly for $t\in[0,\l^{-\frac
  13}]$,
\begin{equation*}
  \sum_{m,n} |c_T(t)|^2\lesssim 1 \,,
  \qquad
  \sum_{m,n} |c'_T(t)|^2\lesssim \l^{\frac 23}\,,
\end{equation*}
which together imply the following bounds, that will then hold
uniformly on each $\l^{-\frac 13}$ time slab,
\begin{equation}\label{cTbounds2}
  \sum_{T:l=l_0} \;\|c_T\|^2_{L^\infty} + \lambda^{-\frac13}\| c'_T\|^2_{L^2} 
  \lesssim 1\,. 
\end{equation}
We apply this expansion separately to the solution $u$ on each
$\l^{-\frac 13}$ time slab, and obtain the full tube decomposition
$u=\sum_T c_T v_T$, where if $T=(l,m,n)$, the functions $c_T$ and
$v_T$ are supported by $t\in I_T\equiv [l\l^{-\frac
  13},(l+1)\l^{-\frac 13}]$.


{\bf Step 3:} {\em Interval decomposition according to packet size.}
Here, given a coefficient $c_T$, we partition the time interval $I_T$
into smaller dyadic subintervals where the coefficient $c_T$ is
essentially constant.  This is done according to the following lemma.
\begin{lemma}\label{cTlemma}
  Let $c:I \to \C$ with
  \[
  \| c\|^2_{L^\infty(I)} + |I|\cdot\|c'\|^2_{L^2(I)} = B\,.
  \]
  Given $\epsilon>0$, there is a partition of $I$ into dyadic
  sub-intervals $I_j$, for each of which either
  \begin{equation}
    \| c\|^2_{L^\infty(I_j)} \geq 4 |I_j| \cdot \| c'\|^2_{L^2(I_j)}
    \quad\text{or}\quad \| c\|_{L^\infty(I_j)}<\epsilon  \,.
    \label{iselect}
  \end{equation}
  Independent of $\epsilon$, the following bound holds
  \[
  \sum_{j} |I_j|^{-1} \| c\|_{L^\infty(I_j)}^2 \le 16B|I|^{-1}\,.
  \]
\end{lemma}
\begin{proof}
  If the test \eqref{iselect} holds on $I$ then no partition is
  needed.  Otherwise we divide the interval in half and retest. The
  test automatically is true if $|I_j|\le \epsilon^2B^{-1}|I|/4$. The
  sum bound then holds by comparing the sum to the $L^2$ norm of $c'$
  over the parent intervals of the $I_j$, which have overlap at most
  2.
\end{proof}

We will take $\epsilon$ to be $\l^{-\frac 13}$.  For each $T$, this
gives a finite partition of its corresponding time interval $I_T$ into
dyadic subintervals, 
\[
I_T = \bigcup_j I_{T,j}
\]
so that \eqref{iselect} holds for $c_T$ in each subinterval,
\begin{equation}
  \| c_T\|^2_{L^\infty(I_{T,j})} \geq 4 |I_{T,j}|\cdot \| c'_{T}\|^2_{L^2(I_{T,j})} \quad
  \text{or}\quad  \| c_T\|_{L^\infty(I_{T,j})} < \l^{-\frac 13}\,,
  \label{iselecta}\end{equation}
and we have the square summability relation
\begin{equation}\label{cTbounds3}
  \sum_j \lambda^{-\frac13}
  |I_{T,j}|^{-1} \| c_T \|_{L^\infty(I_{T,j})}^2 \lesssim 
  \|c_T\|^2_{L^\infty} + \lambda^{-\frac13}\| c'_T\|^2_{L^2}\,.
\end{equation}
We introduce the notation
$$
c_{T,j}=1_{T,j}c_T\,,\qquad c'_{T,j}=1_{T,j}c'_T\,.
$$
Using these interval decompositions, we partition the function $u$ on
$[0,1]\times\R$ into a dyadically indexed sum
\begin{equation}\label{udecomp}
  u = \sum_{a \leq 1} \sum_{k \ge 0} u_{a,k}+u_\epsilon\,,
\end{equation}
where the index $a$ runs over dyadic values between
$\epsilon=\l^{-\frac 13}$ and $1$, and
\[
u_{a,k} = \sum_{(T,j) \in \T_{a,k}} c_{T,j} v_T\,,
\]
with
\[
\T_{a,k} = \{ (T,j):|I_{T,j}| = 2^{-k}\lambda^{-\frac13}, \
\|c_T\|_{L^\infty(I_{T,j})}\in (a,2a] \,\}\,.
\]
We call the functions $u_{a,k}$ above $(a,k)$-packets, and note that,
as the first condition in \eqref{iselecta} holds if $(T,j)\in
\T_{a,k}$ since $a\ge\l^{-\frac 13}$, then
\begin{equation*}
  \frac a4\le |c_T(t)|\le a\,,\quad t\in I_{T,j}\,.
\end{equation*}
We will separately bound in $L^8$ each of the functions $u_{a,k}$.
Since there are at most $\l^{\frac 13}$ tubes $T$ over any point, and
$|v_T|\lesssim\l^{\frac 13}$, we see that
$\|u_\epsilon\|_{L^\infty}\lesssim\l^{\frac 13}$.  On the other hand,
since the decomposition \eqref{udecomp} is almost orthogonal, we have
$\|u_\epsilon\|_{L^2}\lesssim 1$, and hence
$\|u_\epsilon\|_{L^8}\lesssim\l^{\frac 14}$, as desired. For $p>8$ the
bounds on $u_\epsilon$ are even better than needed.

We note here that, by \eqref{cTbounds2} and \eqref{cTbounds3},
$$
\sum_{(T,j)\in\T_{a,k}} \|c_{T,j}\|_{L^2}^2\lesssim 2^{-2k}\,,
$$
hence $\|u_{a,k}\|_{L^2([0,1]\times\R)}\lesssim 2^{-k}$.  \medskip


{\bf Step 4:} {\em Localization weights and bushes.}  To measure the
size of each packet $v_T$ we introduce a bump function in
$I_T\times\R$, namely
\[
\chi_T(t,x) = 1_{I_T}(t)\,\bigl(\,1+\lambda^{\frac23}
|x-x_{T}(t)|\,\bigr)^{-2}\,.
\]
To measure the local density of $(a,k)$-packets we introduce the
function
\[
\chi_{a,k} = \sum_{(T,j) \in \T_{a,k}} 1_{I_{T,j}} \chi_{T}\,.
\]
We note that
\[
|v_T| \lesssim \lambda^{\frac13} \chi_{T}\,,
\]
therefore we have the straightforward pointwise bound
\begin{equation*}
  |u_{a,k}| \lesssim \lambda^{\frac13} a \,\chi_{a,k}\,.
\end{equation*}
This suffices in the low density region
\[
A_{a,k,0} = \{ \chi_{a,k} \le 1 \}\,,
\]
as interpolating the above pointwise bound with the above estimate $\|
u_{a,k}\|_{L^2} \lesssim 2^{-k} $, we obtain
\begin{equation*}
  \| u_{a,k}\|_{L^8(A_{a,k,0})} \lesssim  \lambda^{\frac14} a^{\frac34} 2^{-\frac k4}\,.
\end{equation*}
We may sum over $k\ge 0$ and $a\le 1$ to obtain the desired $L^8$
bound without log factors.  For $p>8$ the resulting bound is even
better than needed.

To obtain bounds over sets where $\chi_{a,k}$ is large, we need to
consider how the solution $u$ behaves on regions larger than a single
$\l^{-\frac 13}$ slab, in addition to more precise bounds within each
slab.


{\bf Step 5:} {\em Concentration scales and bushes.}  Here we
introduce a final parameter $m\ge 1$ which measures the dyadic size of
the packet density. Precisely, we consider the sets
\[
A_{a,k,m} = \{ (t,x)\in [0,1]\times\R: 2^{m-1}< \chi_{a,k}(t,x) \le
2^m \}\,.
\]
The points in $A_{a,k,m}$ are called $(a,k,m)$-bush centers, since as
shown in the next section they correspond to the intersection at time
$t$ of about $2^m$ of the packets comprising $u_{a,k}$.

We remark that by fixed-time $L^2$ bounds on $u$, and tube overlap
considerations, the parameter $m$ must satisfy
\begin{equation}
  2^m a^2 \lesssim 1\,, \qquad 2^m \lesssim \lambda^{\frac13}\,.
  \label{msize}\end{equation}

Our goal will be to bound $\|u_{a,k}\|_{L^8( A_{a,k,m})}$.  Two
considerations guide the proof of our bound.

We first note that a collection of $2^m$ tubes that overlap at a
common time $t$, which we call a $2^m$-bush, can retain full overlap
for time $\delta t = 2^{-m}\l^{-\frac 13}$.  For this to happen the
tubes in the bush must have close angles; if the bush is more spread
out then the overlap time decreases.

On the other hand, for certain Lipschitz metrics like the examples in
\cite{SmSo} and \cite{SmTa}, a focused $2^m$-bush may come back
together after time $\delta t = 2^m\l^{-\frac 13}$.  This indicates
that beyond this scale our only available tool is summation with
respect to the number of such time intervals. Indeed, for each $m$ the
$L^8$ estimate \eqref{l8} is saturated (except for the factors of
$\log\l$) by such a periodically repeating $2^m$-bush.


Given these considerations, we decompose the unit time interval
$[0,1]$ into a collection $\I_m$ of intervals of size $\delta t =
2^{-m}\l^{-\frac 13}$; such intervals are then dyadic subintervals of
the decomposition into $\l^{-\frac 13}$ time slices made in step 2.
The proof of Theorem \ref{l8theorem} is concluded using the following
two propositions. The first one counts how many of these slices may
contain $(a,k,m)$-bushes.

\begin{proposition}
  There are at most $\approx\lambda^{\frac13} 2^{-3m}
  a^{-4}\bigl\langle\log (2^{m} a^{2})\bigr\rangle^3 $ intervals
  $I\in\I_m$ which intersect $ A_{a,k,m}$.
  \label{p:bushcount}\end{proposition}

The second one estimates $\|u_{a,k}\|_{L^8( A_{a,k,m})}$ on a single
$2^{-m}\l^{-\frac 13}$ time slice.

\begin{proposition}
  For each interval $I\in\I_m$, we have
  \begin{equation}\label{e:shorttime}
    \| u_{a,k}\|_{L^8( A_{a,k,m} \cap I\times\R)} \lesssim \lambda^{\frac{5}{24}} 
    2^{\frac{3m}8} a^{\frac12} 2^{-\frac k4}\,.
  \end{equation}
  \label{p:shorttime}\end{proposition}

Combining the two propositions we obtain
\[
\| u_{a,k}\|_{L^8( A_{a,k,m})} \lesssim \lambda^{\frac14}
\bigl\langle\log(2^{m}a^{2})\bigr\rangle^{\frac 38}\,2^{-\frac k4}\,.
\]
The sets $A_{a,k,m}$ are disjoint, and
$\langle\log(2^{m}a^{2})\bigr\rangle\lesssim\log\l$.  Since there are
at most $\log\l$ values of $m$, we obtain
$$
\|u_{a,k}\|_{L^8([0,1]\times \R)}\lesssim \l^{\frac 14}(\log\l)^{\frac
  12}2^{-\frac k4}\,.
$$
We may sum over $k\ge 0$ without additional loss, and there are at
most $\log\l$ distinct values of $a$, which yields the desired
conclusion \eqref{l8}.

For $p>8$, we interpolate \eqref{e:shorttime} with $|u_{a,k}|\lesssim
\l^{\frac 13}2^ma$ to obtain
\begin{equation*}
  \| u_{a,k}\|_{L^p(A_{a,k,m}\cap I\times\R)}^p \lesssim \lambda^{\frac p3-1} \,
  2^{m(p-5)}a^{p-4}\,2^{-2k}\,,
\end{equation*}
and summing over intervals yields
\begin{align*}
  \| u_{a,k}\|_{L^p( A_{a,k,m})}^p &\lesssim
  \lambda^{\frac{p-2}3}\,2^{m(p-8)}a^{p-8}\,2^{-2k}\bigl\langle\log(2^{m}a^{2})\bigr\rangle^3\\
  &=\l^{\frac p2 -2}(2^\frac m2 a)^{p-8}\,(\l^{-\frac 16}2^{\frac
    m2})^{p-8}\,2^{-2k}\,
  \bigl\langle\log(2^{m}a^{2})\bigr\rangle^3\,.
\end{align*}
By \eqref{msize}, the quantity $a$ takes on dyadic values less than
$2^{-\frac m2}$, whereas $2^m$ takes on dyadic values less than
$\l^{\frac 13}$. We may thus sum over $a,k,m$ to obtain the desired
bound
$$
\|u\|_{L^p([0,1]\times\R)}^p\lesssim \l^{\frac p2 - 2}\,,
$$
which together with the estimate for $p=8$ concludes the proof of
Theorem \ref{theorem1.1}.  \qed


\newsection{Bush counting}\label{sec:bushcount} In this section we
reduce the proof of Proposition~\ref{p:bushcount} to Lemma
\ref{l:packet} below.  There are $2^m \l^\frac13$ intervals in $\I_m$,
so the bound is trivial unless $a \ge 2^{-m}$.  We fix a small number
$\epsilon$, to be determined, and consider $\epsilon \,2^{3m}
a^2\langle \log (2^{m} a^{2})\rangle^{-1}$ consecutive intervals in
$\I_m$. Letting $I_\epsilon$ denote their union, then
$$
|I_\epsilon|=\epsilon\,\l^{-\frac 13}2^{2m}
a^2\langle\log(2^ma^2)\rangle^{-1}\,.
$$
It suffices to prove that, if $I_\epsilon$ contains $M$ intervals from
$\I_m$ which intersect $A_{a,k,m}$, then
\begin{equation}\label{Mbound}
  M\lesssim (2^m a^2)^{-1} \langle \log (2^{m} a^{2})\rangle^2\,.
\end{equation}


Heuristically we would like to say that a point in $A_{a,k,m}$
corresponds to $2^m$ packets through a point. To make this precise, we
need to take into account the tails in the bump functions $\chi_{T}$.

Consider a point $(t,x)$ in a $2^{-m}\l^{-\frac 13}$ slice
$I\times\R$, such that $\chi_{a,k}(t,x)\ge 2^m$.  For each $y\in\R$,
we denote by $N(y)$ the number of tubes $T$ in the definition of
$\chi_{a,k}$ which are centered near $y$ at time $t$, i.e.
\[
N(y)= \#\bigl\{(T,j)\in \T_{a,k}\,:\,t\in
I_{T,j}\;\;\text{and}\;\;|x_T(t)-y|\le\l^{-\frac 23}\bigr\}\,.
\]
Then
\[
\chi_{a,k}(t,x) \lesssim \lambda^{\frac23} \int \bigl(1+
\lambda^{\frac23}|x-y|\bigr)^{-2} N(y)\, dy\,.
\]
Hence there must exist some point $y$ such that $N(y) \gtrsim
2^m$. Thus, we can find a point $y$ and $\gtrsim 2^m$ indices
$(T,j)\in\T_{a,k}$ for which $|x_T(t)-y|\le \l^{-\frac 23}$ and $t\in
I_{T,j}$.  Since there are at most 5 values of $T$ with the same
$\xi_T$ for which $|x_T(t)-y|\le\l^{-\frac 23}$, we may select a
subset of $\approx 2^m$ packets which have distinct values of $\xi_T$.
We call this an \textit{$(a,k,m)$-bush} centered at $(t,y)$. For
simplicity, we assume the bush contains exactly $2^m$ terms.

Consider then a collection $\{B_n\}_{n=1}^M$ of $M$ distinct
$(a,k,m)$-bushes, centered at $(t_n,x_n)$, with
\begin{equation}\label{epsdef}
  \epsilon\,\l^{-\frac 13}2^{2m} a^2\langle\log(2^ma^2)\rangle^{-1}
  \ge|t_n-t_{n'}|\ge \l^{-\frac 13}2^{-m}\quad\text{when}\quad n\ne n'\,.
\end{equation}
Denote by $\{v_{n,l}\}_{l=1,2^m}$ the collection of $2^m$ terms $v_T$
comprising $B_n$.  For each $n$ we define
$$
w_n=a\sum_l \overline{c_{n,l}}(t_n)^{-1} v_{n,l}(t_n, \cdot\,)\,.
$$
Since $|c_{n,l}(t_n)| \approx a$, and the $\xi_T$ are distinct, then
$\|w_n\|_{L^2(\R)}\approx 2^{\frac m2}$, by \eqref{Gaborsum} and the
fact that $v_T(t_n,\cdot)$ is the image of $\phi_T$ under the unitary
flow of $a^w_{\l^{2/3}}(t,x,D)$.

We then define the approximate projection operators $P_n$ on $L^2(\R)$
by
\[
P_n f = 2^{-m} \langle w_n,f\rangle \,w_n\,.
\]
With $u$ the solution to \eqref{directional-hom} we recall that
$\langle v_{n,l}(t,\cdot\,),u(t,\cdot\,)\rangle=c_{n,l}(t)$. Applying
$P_n$ to $u$ at time $t_n$, we then have
\[
P_n u = a\,w_n\,,
\]
therefore
\begin{equation}
  \| P_n u\|_{L^2(\R)}^2 \approx 2^m a^2\,.
  \label{ag}\end{equation}

If these projectors were orthogonal with respect to the flow of
\eqref{directional-hom}, that is
\[
P_{n'} S(t_{n'},t_{n}) P_{n} = 0\,,
\]
then we would obtain
\[
1 = \|u_0\|_{L^2(\R)}^2 \gtrsim \sum_{n} \| P_n u\|_{L^2(\R)}^2
\approx M 2^m a^2\,,
\]
and \eqref{Mbound} would be trivial. This is too much to hope
for. Instead, we will prove that the operators $P_n$ satisfy an almost
orthogonality relation:
\begin{lemma}\label{l:packet}
  Let $\alpha=\max\bigl(\l^{-\frac 13}|t_{n'}-t_n|^{-1},\l^{\frac
    13}|t_{n'}-t_n|\bigr)$.  Then the operators $P_n$ satisfy
  \begin{equation}
    \| P_{n'} S(t_{n'},t_n) P_n \|_{L^2(\R) \to L^2(\R)} \lesssim 
    { 2^{-m}\alpha }\,{ \langle\log (2^{-m}\alpha)\rangle }\,.
    \label{ao}
  \end{equation}
  \label{l:packet}\end{lemma}
We postpone the proof of Lemma \ref{l:packet} to the end of section
\ref{sec:wavepacket}.  This estimate is not strong enough to allow us
to use Cotlar's lemma. However, we can prove a weaker result, namely
that for any solution $u$ to \eqref{directional-hom} we have
\begin{equation}
  \sum_n \| P_n u\|_{L^2(\R)} \lesssim   
  C^{\frac12}
  \| u_0\|_{L^2(\R)}\,, \qquad C =  M+ \sum_{n,n'}
  { 2^{-m}\alpha } \,\langle\log (2^{-m}\alpha)\rangle \,.
  \label{ug}\end{equation}
Indeed, by duality \eqref{ug} is equivalent to
\[
\Bigl\|\, \sum_n S(0,t_n) P_n f_n \,\Bigr\|_{L^2(\R)} \lesssim
C^{\frac12} \sup_n \|f_n\|_{L^2(\R)}\,.
\]
Using \eqref{ao} we have
\[
\begin{split}
  \Bigl\|\, \sum_n S(0,t_n) P_n f_n\, \Bigr\|_{L^2(\R)}^2 =&
  \,\sum_{n,n'} \langle f_{n'}, P_{n'} S(t_{n'},t_n) P_n f_n \rangle
  \\ \lesssim & \, \Bigl(M+ \sum_{n\ne n'}{ 2^{-m}\alpha
  }\,\langle\log (2^{-m}\alpha)\rangle \Bigr) \sup_n
  \|f_n\|_{L^2(\R)}^2\,,
\end{split}
\]
establishing \eqref{ug}.

Comparing \eqref{ag} and \eqref{ug} applied to $u$, it follows that
\[
2^{\frac m2}aM\lesssim C^{\frac 12}\,,
\]
or, in expanded form,
\begin{equation}\label{Msumbound}
  2^m a^2 M^2 \lesssim  M+ \sum_{n\ne n'}{ 2^{-m}\alpha }\,{ \langle\log (2^{-m}\alpha)\rangle }
  \,.
\end{equation}
This will be the source of our bound in \eqref{Mbound} for $M$.  If
\[
2^m a^2 M^2 \lesssim M
\]
then we are done, so we consider the summation term.

First consider the sum over terms where $\l^{\frac 13}|t_{n'}-t_n|\ge
1$, for which we have
$$
\sum_{n\ne n'}2^{-m}\l^{\frac 13}|t_{n'}-t_n|\, \bigl\langle\log
(2^{-m}\l^{\frac 13}|t_{n'}-t_n|\,)\bigr\rangle \lesssim
\epsilon\,|\log\epsilon|\, 2^m a^2 M^2\,,
$$
where we use that $r\langle \log r\rangle$ is an increasing function,
and by \eqref{epsdef} and \eqref{msize} that
$$
2^{-m}\l^{\frac 13}|t_{n'}-t_n|\le \epsilon \, 2^{m}a^2\langle\log(2^m
a^2)\rangle^{-1}\lesssim \epsilon\,.
$$
Taking $\epsilon$ small we can thus absorb these terms into the left
hand side of \eqref{Msumbound}.

To conclude the proof, we consider the sum over $\l^{\frac
  13}|t_{n'}-t_n|\le 1$. By the $2^{-m}\l^{-\frac 13}$ separation of
the $M$ points $t_n$, we have
$$
\sum_{n\ne n'} {2^{-m}\l^{-\frac 13}|t_{n'}-t_n|^{-1}} \,
{\bigl\langle\log (2^m\l^{\frac 13}|t_{n'}-t_n|)\bigr\rangle} \lesssim
M(\log M)^2\,.
$$
We conclude that
\[
2^m a^2 M \lesssim (\log M)^2\,,
\]
hence that
\[
M \lesssim (2^{m} a^{2})^{-1} \bigl\langle\log (2^{m}
a^{2})\bigr\rangle^{2}\,.
\]


\newsection{Short time bounds}\label{sec:shorttime} In this section we
reduce the proof of Proposition~\ref{p:shorttime} to a combination of
weighted Strichartz estimates and bilinear estimates, which are proved
respectively in Sections \ref{sec:wavepacket} and \ref{sec:bilinear}.
We remark that the use of Strichartz estimates alone leads to
Proposition \ref{p:shorttime'} instead, and for $d=2$ this yields the
estimates of \eqref{qboundc} only for $p>10$.

We recall the bound we need,
\begin{equation*}
  \| u_{a,k}\|_{L^8(A_{a,k,m}\cap I\times\R)} \lesssim 
  \lambda^{\frac{5}{24}}    2^{\frac38m} a^{\frac12} 2^{-\frac k4}
\end{equation*}
where $|I|=2^{-m}\l^{-\frac 13}$.  Here $u_{a,k}$ on $I\times\R$ has
the form
\begin{equation*}
  1_I(t)\cdot u_{a,k} = 
  \sum_{(T,j)\in\T_{a.k}\,:\,I_{T,j}\cap I\ne\emptyset}1_{I}c_{T,j} v_T \,,
\end{equation*}
where we recall that $c_{T,j}=1_{I_{T,j}}c_T$,
$c'_{T,j}=1_{I_{T,j}}c'_T$.  Also,
$|I_{T,j}|=2^{-k}\lambda^{-\frac13}$, and
\[
|c_{T,j}|\approx a\,, \qquad \|c'_{T,j}\|_{L^2} \lesssim
\lambda^{\frac16}2^{\frac k2}a\,.
\]
Note that if $k\le m$, then $I_{T,j}\supseteq I$ for each term in the
sum, whereas if $k>m$, then $I_{T,j}$ is a dyadic subinterval of $I$,
and there may be multiple terms associated to a tube $T$.

We let $N$ denote the number of terms in the sum for $u_{a,k}$, and
note that, by \eqref{cTbounds2} and \eqref{cTbounds3}, we have
\begin{equation}\label{Nakbound}
  N a^2 \lesssim 2^{-k}\,.
\end{equation}
Using this bound, and dividing $u_{a,k}$ by $a$, we then need
establish the following.

\begin{lemma}\label{l8blemma} 
  Let $\T$ be a collection of $N$ distinct pairs $(T,j)$, and
  $I_{T,j}$ corresponding intervals of length $2^{-k}\l^{-\frac 13}$
  which intersect the interval $I$ of length $2^{-m}\l^{-\frac
    13}$. Assume that
$$
\|c_{T,j}\|_{L^\infty}+2^{-\frac k2}\l^{-\frac
  16}\|c'_{T,j}\|_{L^2}\le 1\,.
$$
Then with $v=\sum_{\T} c_{T,j} v_T$, the following holds
\[
\| v\|_{L^8(A_m\cap I\times\R)} \lesssim \lambda^{\frac{5}{24}}
N^{\frac14} 2^{\frac38m}\,,
\]
where
\[
A_m = \{ \chi \approx 2^m \}\,,\qquad \chi =\sum_{T\in\T} \chi_T\,.
\]
\end{lemma}

To start the proof, we first show that we can dispense with the high
angle interactions. We want to establish
\[
\| v^2\|_{L^4(A_m\cap I\times\R)} \lesssim \lambda^{\frac{5}{12}}
N^{\frac 12} 2^{\frac34m}\,.
\]
We express $v^2$ using a bilinear angular decomposition.  Fixing some
reference angle $\theta$ we can write
$$
v^2 = \sum_{l} \pm\, v_l^2 \,\;+ \sum_{\angle(T,S) \ge \theta}c_{T,j}
v_T \cdot c_{S,k} v_S
$$
where $v_l=\sum_{\xi_T\in K_l}c_{T,j}v_T$ consists of the terms for
which $\xi_T$ lies in an interval $K_l$ of length $\approx\l\theta$,
where the $K_l$ have overlap at most 3. The second sum is over a
subset of $\T\times\T$ subject to the condition
$\angle(T,S)=\l^{-1}|\xi_T-\xi_S|\ge \theta$.

For the second term we have a bilinear $L^2$ estimate,

\begin{lemma}\label{bilinearlemma}
  The following bilinear $L^2$ bound holds,
  \begin{equation}
    \biggl\| \,\sum_{\angle(T,S) \ge \theta} c_{T,j} v_T\cdot  c_{S,k} v_S \,\biggr\|_{L^2(I\times\R)}
    \lesssim
    \theta^{-\frac12}\sum_{T\in\T}\,\Bigl(\, \|c_{T,j}\|_{L^\infty}^2+
    2^{-k}\l^{-\frac13}\|c'_{T,j}\|_{L^2}^2 \,\Bigr)
    \,.
    \label{l2bi}\end{equation}
\end{lemma}
For this estimate there is no restriction on the number of tubes, nor
do we require equal size of the $c_T$. The integral can furthermore be
taken over the interval of length $\l^{-\frac 13}$ containing $I$; the
short time condition is needed only for the small angle interactions.

We prove \eqref{l2bi} in section \ref{sec:bilinear}. In our case, each
term in the sum on the right is bounded by 1.  On the other hand, we
have $|v_T| \lesssim \lambda^\frac13 \chi_T$ therefore
\[\, \biggl| \,\sum_{\angle(T,S) \ge \theta} c_{T,j} v_T \cdot c_{S,k}
v_S\, \biggr| \lesssim \lambda^\frac23\chi^2\,,
\]
which yields
\[
\biggl\| \,\sum_{\angle(T,S) \ge \theta} c_{T,j} v_T \cdot c_{S,k}
v_S\, \biggr\|_{L^\infty(A_m)} \lesssim \lambda^\frac23 2^{2m}
\]
Interpolating the $L^2$ and the $L^\infty$ bounds we obtain
\[
\biggl\| \, \sum_{\angle(T,S) \ge \theta} c_{T,j} v_T \cdot c_{S,k}
v_S\,\biggr\|_{L^4(A_m)} \lesssim \lambda^{\frac13} N^\frac12 2^m
\theta^{-\frac14} \,.
\]
This is what we need for the high angle component provided that
\[
\lambda^{\frac13} N^\frac12 2^m \theta^{-\frac14} =
\lambda^\frac{5}{12} N^\frac12 2^{\frac34m}\,,
\]
or equivalently
\[
\theta = \lambda^{-\frac13} 2^m\,.
\]

Hence it suffices to restrict ourselves to the terms $v_l$, where
$\xi_T\in K_l$, an interval of width $\delta \xi = \lambda^{\frac23}
2^m$ centered on $\xi_l$.  Let $N_l$ denote the number of terms
$(T,j)$ in $v_l$, so that $\sum_l N_l\le 3N$.  We will prove that
\begin{equation}\label{vlbounds}
  \bigl\|(1+\l^{-\frac 43}2^{-2m}| D-2\xi_l|^2) v_l^2\bigr\|_{L^3(I\times\R)}\lesssim 
  \l^\frac13 N_l^\frac23 2^{\frac m3 }\,.
\end{equation}
Let $Q_l(D)=(1+\l^{-\frac 43}2^{-2m}| D-2\xi_l|^2)^{-1}$. We observe
that, for $w_l\in\mathcal S(\R)$,
$$
\Bigl\|\,\sum_l Q_l w_l\,\Bigr\|_{L^3(\R)}\lesssim
\Bigl(\sum_l\|w_l\|_{L^3(\R)}^{\frac 32}\Bigr)^\frac23\,,
$$
which follows by interpolating the bounds
$$
\Bigl\|\,\sum_l Q_l w_l\,\Bigr\|_{L^\infty(\R)}\lesssim
\sum_l\|w_l\|_{L^\infty(\R)}\,,\qquad \Bigl\|\,\sum_l Q_l
w_l\,\Bigr\|_{L^2(\R)}\lesssim
\Bigl(\sum_l\|w_l\|_{L^2(\R)}^2\Bigr)^\frac12\,.
$$
The second follows by the finite overlap condition, the first since
$Q_l$ is convolution with respect to an $L^1$ function. Applying this
to \eqref{vlbounds} yields
$$
\Bigl\|\,\sum_l v_l^2\,\Bigr\|_{L^3(I\times\R)}\lesssim \l^\frac 13
N^\frac23 2^\frac m3\,.
$$
Interpolating with the $L^\infty$ bounds as above yields the desired
bound
$$
\Bigl\|\,\sum_l v_l^2\,\Bigr\|_{L^4(A_m\cap I\times\R)} \lesssim
\l^\frac5{12} N^\frac 12 2^{\frac34 m}\,.
$$

By Leibniz' rule and H\"older's inequality, \eqref{vlbounds} follows
from showing, for $n\le 2$, that
\begin{equation}\label{vlbounds1}
  \bigl\|\bigl(\l^{-\frac 23}2^{-m}(D-\xi_l)\bigr)^n v_l\bigr\|_{L^6(I\times\R)}\lesssim 
  \l^\frac16 N_l^\frac13 2^{\frac m6 }\,.
\end{equation}

We first note the following bound on $v_l$ in $L^6$ over the entire
$2^{-\min(k,m)}\l^{-\frac 13}$ time slice $I^*\times\R$ on which the
$v_l$ are supported:
\begin{equation}\label{vlbounds2}
  \bigl\|\bigl(\l^{-\frac 23}2^{-m}(D-\xi_l)\bigr)^n v_l\bigr\|_{L^6(I^*\times\R)}\lesssim 
  \l^\frac16 N_l^\frac 12\,.
\end{equation}
To establish this, it suffices by the generalized Minkowski inequality
to establish it on a time interval $J$ of length $2^{-k}\l^{-\frac
  13}$, with $N_l$ replaced by the number $N_J$ of $(T,j)$ for which
$I_{T,j}=J$.  If $k\le m$, then there is only one interval $J$ to
consider, whereas $k>m$ means $J$ is a dyadic subdivision of $I$.  If
$t_0$ is the left endpoint of $J$, then we have the initial data bound
$$
\|\bigl(\l^{-\frac 23}2^{-m}(D-\xi_l)\bigr)^n
v_l(t_0)\|_{L^2(\R)}\lesssim \Bigl(\sum |c_{T,j}(t_0)|^2\Bigr)^{\frac
  12}\lesssim N_J^\frac 12\,,
$$
and for the inhomogeneous term we have
$$
\|\bigl(\l^{-\frac
  23}2^{-m}(D-\xi_l)\bigr)^n\bigl(D_t-a_{\l^{2/3}}^w(t,x,D)\bigr)v_l
\|_{L^1_tL^2_x(J\times\R)} \lesssim |J|^{\frac 12}\Bigl(\sum
\|c'_{T,j}\|_{L^2(J)}^2\Bigr)^{\frac 12}\lesssim N_J^\frac 12\,.
$$
The result then holds by the weighted Strichartz estimates, Theorem
\ref{weightedStrichartz}.

To obtain the gain in the norm over the slice $I\times\R$, we make a
further decomposition $v_l=\sum_B v_B$ into ``bushes".  This is made
by decomposing the $x$-axis into disjoint intervals of radius
$\l^{-\frac 23}$, indexed by $B$, with center $x_B$, and letting $v_B$
denote the sum of the $c_{T,j} v_T$ in $v_l$ whose center $x_T$ at
time $t_0$ satisfies $|x_T-x_B|\le\l^{-\frac 23}$.

For simplicity, we take $t_0=0$.  Let $x_B(t)$ denote the
bicharacteristic curve passing through $(x_B,\xi_l)$. Then
$|x_T-x_B|\lesssim \l^{-\frac 23}$, provided $T$ is part of $v_B$. By
Theorem \ref{weightedStrichartz} we thus have the weighted Strichartz
estimates,
\begin{equation*}
  \bigl\|
  \bigl(1+\l^{\frac 43}|x-x_B(t)|^2\bigr)
  \bigl(\l^{-\frac 23}2^{-m}(D-\xi_l)\bigr)^n v_B\bigr\|_{L^6(I\times\R)}\lesssim 
  \l^\frac16 N_B^\frac12\,,
\end{equation*}
where $N_B$ is the number of terms in $v_B$.  We may sum over $B$ to
obtain
\begin{equation}\label{vlbounds4}
  \bigl\|\bigl(\l^{-\frac 23}2^{-m}(D-\xi_l)\bigr)^n v_l\bigr\|_{L^6(I\times\R)}\lesssim 
  \l^\frac16 \Bigl(\,\sum_B N_B^3\Bigr)^\frac 16
  \le \l^\frac 16 N_l^\frac 16 2^{\frac m3}
  \,,
\end{equation}
where at the last step we used $N_B\le 2^m$, and $\sum_B N_B=N_l$.
Combining \eqref{vlbounds4} with \eqref{vlbounds2} yields
\eqref{vlbounds1}.  \qed


\newsection{Wave packet propagation}\label{sec:wavepacket} In this
section we establish the basic properties of the wave packet solutions
$v_T$ on the $\l^{-\frac 13}$ time scale, and prove weighted
Strichartz estimates. In addition, we give the proof of
Lemma~\ref{l:packet}.  The results of this section are closely related
to those of \cite[Section 4]{KT} through a space-time rescaling, but
for completeness we provide full proofs.

Throughout this section, we let $A=a^w_{\l^{2/3}}(t,x,D)\,,$ and let
$u$ solve
$$
(D_t-A)u=0\,,\qquad u(0,\cdot\,)=u_0\,.
$$ 
We assume $u_0\in\mathcal{S}$, so that all derivatives of $u$ are
rapidly decreasing in $x$. Throughout, $I$ is an interval with left
hand endpoint $0$ and $|I|\le\l^{-\frac 13}$.

\begin{lemma}\label{freqloc}
  For any $m,n\ge 0$ and $\xi_0\in\R$,
  \begin{equation*}
    \sum_{j=0}^n\bigl(\l^{-\frac 23}2^{-m}\bigr)^j\|(D-\xi_0)^j u\|_{L^\infty L^2(I\times\R)}
    \le C_n
    \sum_{j=0}^n\bigl(\l^{-\frac 23}2^{-m}\bigr)^j\|(D-\xi_0)^j u_0\|_{L^2(\R)}\,.
  \end{equation*}
\end{lemma}
\begin{proof}
  We use induction on $n$.  The case $n=0$ follows by self-adjointness
  of $A$, so we assume the result holds for $n-1$.  We may write the
  commutator
$$
\l^{-\frac 23}[(D-\xi_0),A]=\l^{\frac 13} b^w(t,x,D)\,,\quad b\in
S_{\l,\l^{2/3}}\,,
$$
whereas commuting with $\l^{-\frac 23}(D-\xi_0)$ preserves the set of
Weyl-pseudodifferential operators with symbol in
$S_{\l,\l^{2/3}}$. Hence, we may write
$$
(D_t-A)\bigl(\l^{-\frac 23}2^{-m}\bigr)^n(D-\xi_0)^n u= \l^{\frac
  13}\sum_{j=0}^{n-1} b^w(t,x,D) \bigl(\l^{-\frac
  23}2^{-m}\bigr)^j(D-\xi_0)^j u\,,
$$
where $b\in S_{\l,\l^{2/3}}$ may vary with $j$. The proof follows by
$L^2$ boundedness of $b^w(t,x,D)$ and the Duhamel formula, since
$|I|\le\l^{-\frac 13}$.
\end{proof}

A similar proof, using the fact that
$$
\l^{\frac 13}[x,A]=\l^{\frac 13}b^w(t,x,D)\,,\quad b\in
S_{\l,\l^{2/3}}\,,
$$
and that commuting with $\l^{\frac 13}x$ preserves $S_{\l,\l^{2/3}}$,
yields the following.
\begin{corollary}\label{c:weight}
  For any $l,m,n\ge 0$, and all $x_0,\xi_0\in\R$,
  \begin{multline*}
    \sum_{j=0}^n\sum_{k=0}^l
    \l^{\frac k3 }\bigl(\l^{-\frac 23}2^{-m}\bigr)^j\|(x-x_0)^k(D-\xi_0)^j u\|_{L^\infty L^2(I\times\R)}\\
    \le C_{n,l} \sum_{j=0}^n\sum_{k=0}^l \l^{\frac k3}\bigl(\l^{-\frac
      23}2^{-m}\bigr)^j\|(x-x_0)^k(D-\xi_0)^j u_0\|_{L^2(\R)}\,.
  \end{multline*}
\end{corollary}

To obtain weighted localization in $x$ at the $\l^{-\frac 23}$ scale
we need to evolve the spatial center of $u$ along the bicharacteristic
flow.  Additionally, we must work on a time interval $I$ so that the
spread of bicharacteristics due to the spread of frequency support is
less than $\l^{-\frac 23}$.

\begin{lemma}\label{spacefreqloc}
  Let $x_0(t)=x_0-t\,\partial_\xi a_{\l^{2/3}}(0,x_0,\xi_0)\,,$ and
  suppose that $|I|\le 2^{-m}\l^{-\frac 13}$. Then for $l\le n$, and
  general $m,n,x_0,\xi_0$,
  \begin{multline*}
    \sum_{j=0}^n\sum_{k=0}^l
    \l^{\frac {2k}3 }\bigl(\l^{-\frac 23}2^{-m}\bigr)^j\|(x-x_0(t))^k(D-\xi_0)^j u\|_{L^\infty L^2(I\times\R)}\\
    \le C_n \sum_{j=0}^n\sum_{k=0}^l \l^{\frac {2k}3}\bigl(\l^{-\frac
      23}2^{-m}\bigr)^j\|(x-x_0)^k(D-\xi_0)^j u_0\|_{L^2(\R)}\,.
  \end{multline*}
\end{lemma}
\begin{proof}
  We write
$$
i\,[x-x_0(t),D_t-A]= (\partial_\xi a_{\l^{2/3}})^w(t,x,D)-\partial_\xi
a_{\l^{2/3}}(0,x_0,\xi_0)\,,
$$
and taking a Taylor expansion write
\begin{multline*}
  \l^{\frac23}
  \bigl(\partial_\xi a_{\l^{2/3}}(t,x,\xi)-\partial_\xi a_{\l^{2/3}}(0,x_0,\xi_0)\bigr)=\\
  \l^{\frac 13}2^m\Bigl( b_1(t,x,\xi)\l^{\frac
    13}(x-x_0)+b_2(t,x,\xi)\l^{\frac 13} t +b_3(t,x,\xi)\l^{-\frac
    23}2^{-m}(\xi-\xi_0)\Bigr)\,,
\end{multline*}
with $b_j\in S_{\l,\l^{2/3}}$, where we use $2^m\ge 1$. Additionally,
commuting with $\l^{\frac 23}x$ preserves the class of $b^w(t,x,D)$
with $b\in S_{\l,\l^{2/3}}$. The proof now proceeds along the lines of
the proof of Lemma \ref{freqloc}, using that $|I|\le\l^{-\frac
  13}2^{-m}$.
\end{proof}

We remark that the proof of Lemma \ref{spacefreqloc} in fact shows
that one may bound
\begin{multline*}
  \sum_{j=0}^n\sum_{k=0}^l
  \l^{\frac {2k}3 }\bigl(\l^{-\frac 23}2^{-m}\bigr)^j\|(D_t-A)(x-x_0(t))^k(D-\xi_0)^j u\|_{L^1 L^2(I\times\R)}\\
  \le C_n \sum_{j=0}^n\sum_{k=0}^l \l^{\frac {2k}3}\bigl(\l^{-\frac
    23}2^{-m}\bigr)^j\|(x-x_0)^k(D-\xi_0)^j u_0\|_{L^2(\R)}\,,
\end{multline*}
provided $l\le n$ and $|I|\le 2^{-m}\l^{-\frac 13}$, or
$|I|\le\l^{-\frac 13}$ in case $l=0$. We thus can prove weighted
Strichartz estimates as an easy corollary of the unweighted
version. We state the result for $p=q=6$, but it holds for all
allowable values of $(p,q)$ for which the unweighted version holds.

\begin{theorem}\label{weightedStrichartz}
  Let $x_0(t)=x_0-t\,\partial_\xi a_{\l^{2/3}}(0,x_0,\xi_0)\,,$ and
  suppose that $|I|\le 2^{-m}\l^{-\frac 13}$. Then for $l\le n$, and
  general $m,n,x_0,\xi_0$,
  \begin{multline*}
    \sum_{j=0}^n\sum_{k=0}^l
    \l^{\frac {2k}3 }\bigl(\l^{-\frac 23}2^{-m}\bigr)^j\|(x-x_0(t))^k(D-\xi_0)^j u\|_{L^6(I\times\R)}\\
    \le C_n\,\l^{\frac 16} \sum_{j=0}^n\sum_{k=0}^l \l^{\frac
      {2k}3}\bigl(\l^{-\frac 23}2^{-m}\bigr)^j\|(x-x_0)^k(D-\xi_0)^j
    u_0\|_{L^2(\R)}\,.
  \end{multline*}
  If $l=0$, then the result holds for $|I|\le\l^{-\frac 13}$.
\end{theorem}
\begin{proof}
  By the above remarks, the result follows by the Duhamel theorem from
  the case $n=l=0$.  That case, in turn, follows from \cite[Theorem
  2.5]{KT}. An alternate proof is contained in \cite{BSS}. The paper
  \cite{BSS} dealt with $\l^{-1}\Delta_\g$ instead of $A$, but the
  analysis is similar for $A$ as above.
\end{proof}

If we take $m=0$, then Lemma \ref{spacefreqloc} applies to the
evolution of a $\l^{-\frac 23}$ packet. The following should be
compared to \cite[Proposition 4.3]{KT}.
\begin{theorem}\label{tubeloc}
  Suppose that $\phi$ is a Schwartz function, and $\phi_T=\l^{\frac
    13}e^{ix\xi_T}\phi(\l^{\frac 23}(x-x_T))$.  Let $v_T$ satisfy
$$
\bigl(D_t - A\bigr) v_T = 0 \,, \qquad v_T(0,\cdot\,) = \phi_T\,.
$$
Then with
$$
x_T(t)=x_T-t\,\partial_\xi a_{\l^{2/3}}(0,x_T,\xi_T)\,,
$$
for $t\in[0,\l^{-\frac 13}]$ one can write
$$v_T(t,\cdot)=
\l^{\frac 13}e^{ix\xi_T}\psi_T\bigl(t,\l^{\frac 23}(x-x_T(t))\bigr)\,,
$$
where $\{\psi_T(t,\cdot)\}_{t\in I}$ is a bounded family of Schwartz
functions on $\R$, with all Schwartz norms uniformly bounded over $T$,
$t\in I$, and $\lambda\ge 1$.
\end{theorem}


We conclude this section with the proof of Lemma~\ref{l:packet}.  Let
$P_0$ denote the bounded linear functional on $L^2(\R)$ defined by
$P_0f=2^{-m}\langle w_0,f\rangle w_0$, where $w_0$ is a sum of $2^m$
$L^2$-bounded packets centered at $x_0$ with disjoint frequency
centers; that is,
$$
w_0=\sum_{l=1}^{2^m} \l^{\frac 13}e^{ix\xi_l}\psi_l\bigl(\l^{\frac
  23}(x-x_0)\bigr)\,,
$$
where the $\psi_l$ are a bounded collection of Schwartz functions, and
the $\xi_l$ are distinct points on the $\l^{\frac 23}$-spaced lattice
in $\R$, with $|\xi_l|\le \frac 12\l$.  We take $P_1$ to be similarly
defined where $x_0$ is replaced by $x_1$, possibly with a different
set of $\xi_l$ and $\psi_l$.  We need to prove
\begin{equation}
  \| P_1 S(t_1,t_0) P_0\|_{L^2(\R) \to L^2(\R)} \lesssim 
  { 2^{-m}\alpha }\,{ \langle\log (2^{-m}\alpha)\rangle }\,,
  \label{ao2}
\end{equation}
where $\alpha=\max\bigl(\l^{-\frac 13}|t_1-t_0|^{-1},\l^{\frac
  13}|t_1-t_0|\bigr)$.

The proof of \eqref{ao2} requires control of the solution $u$ over
times greater than $\l^{-\frac 13}$, which we will express in the form
of weighted energy estimates. Heuristically, for $t\le\l^{-\frac 13}$
one can localize energy flow at the symplectic $\l^{\frac 23}$
scale. For $t\ge\l^{-\frac 13}$, energy flow cannot be localized finer
than the uncertainty in the Hamiltonian flow, where $\xi$ is
determined only within $\l t$, with resulting uncertainty in $x$.  In
the notation below, our weighted energy estimate localizes $\xi$ to
within $\delta\l$, and $x$ to within $\delta^2$. The linear growth of
the weights reflects the Lipschitz regularity of $a_\l(t,x,\xi)$.

In the following, we let
$$
\delta=\begin{cases}\l^{-\frac 13}\,,& |t_0-t_1|\le \l^{-\frac 13}\,,\\
  |t_0-t_1|\,,& |t_0-t_1|\ge \l^{-\frac 13}\,.
\end{cases}
$$
Let $q_j(t,x,D)$, $j=0,1$, denote the symbol
$$
q_j(t,x,\xi)=\delta^{-2}\bigl(x-x_j+(t-t_j)\,\partial_\xi
a_\l(t_j,x_j,\xi)\bigr)\,,
$$
and set $Q_j(t)=q_j(t,x,D)$.  We will prove that
\begin{equation}\label{weightedest}
  \|Q_0(t_1)S(t_1,t_0)f\|_{L^2(\R)}\lesssim \|f\|_{L^2(\R)}
  +\|Q_0(t_0)f\|_{L^2(\R)} +\delta^{-1}\|(x-x_0)f\|_{L^2(\R)}\,.
\end{equation}

Assuming \eqref{weightedest} for the moment, we consider the constant
coefficient symbols
$$
m_0(\xi)=\langle q_0(t_1,x_1,\xi)\rangle\,,\qquad\quad
m_1(\xi)=\langle q_1(t_0,x_0,\xi)\rangle\,.
$$
We will use \eqref{weightedest} to prove that
\begin{equation}\label{weightedest2}
  \bigl\|m_0(D)^{\frac 12}\langle\delta^{-2}(x-x_1)\rangle^{-2}S(t_1,t_0)
  \langle\delta^{-2}(x-x_0)\rangle^{-2}
  m_1(D)^{\frac 12}f\bigr\|_{L^2(\R)}\lesssim \|f\|_{L^2(\R)}\,,
\end{equation}
with bounds uniform over the various parameters.  We then factor
\begin{multline*}
  P_1S(t_1,t_0)P_0=
  P_1\langle\delta^{-2}(x-x_1)\rangle^{2}m_0(D)^{-\frac
    12}m_0(D)^{\frac 12}
  \langle\delta^{-2}(x-x_1)\rangle^{-2}\\
  S(t_1,t_0) \langle\delta^{-2}(x-x_0)\rangle^{-2} m_1(D)^{\frac
    12}m_1(D)^{-\frac 12}\langle\delta^{-2}(x-x_0)\rangle^{2}P_0\,,
\end{multline*}
which reduces \eqref{ao2} to showing that
\begin{equation}\label{weightedest3}
  \bigl\|P_1\langle\delta^{-2}(x-x_1)\rangle^{2}m_0(D)^{-\frac 12}f\bigr\|_{L^2(\R)}^2\lesssim 
  { 2^{-m}\alpha }\,{ \langle\log (2^{-m}\alpha)\rangle }\|f\|_{L^2(\R)}^2\,,
\end{equation}
where we use symmetry and adjoints to conclude that the rightmost
factor above satisfies the same bounds. Since $P_1 f=2^{-m}\langle
w_1,f\rangle w_1$, and $\|w_1\|_{L^2(\R)}\approx 2^{\frac m2}$,
\eqref{weightedest3} is implied by
\begin{equation}\label{wbound}
  \bigl\|m_0(D)^{-\frac 12}\langle\delta^{-2}(x-x_1)\rangle^{2}w_1\|_{L^2(\R)}^2\lesssim
  \alpha\,{ \langle\log (2^{-m}\alpha)\rangle }\,.
\end{equation}
To prove \eqref{wbound}, note that since $\delta\ge\l^{-\frac 13}$,
and the packets in $w_1$ are centered at $x_1$, the function
$\langle\delta^{-2}(x-x_1)\rangle^{2}w_1$ is of the same form as
$w_1$. The Fourier transform of $w_1$ is a sum of $2^m$
$L^2$-normalized Schwartz functions, concentrated on the $\l^{\frac
  23}$-scale about the distinct $\xi_l$, and the left hand side of
\eqref{wbound} can thus be compared to
$$
\l^{-\frac 23}\int\,\Bigl|\,\sum_l m_0(\xi)^{-\frac
  12}\bigl(1+\l^{-\frac 23}|\xi-\xi_l|\bigr)^{-N} \Bigr|^2\,d\xi
\lesssim \sum_l m_0(\xi_l)^{-1}\,.
$$
By \eqref{acond1},
$$
\partial_\xi a_\l(t_0,x_0,\xi_l)-\partial_\xi
a_\l(t_0,x_0,\xi_{l'})\approx \l^{-1}(\xi_{l'}-\xi_{l})\,.
$$
Since $\l^{\frac 13}\delta^{2}|t_1-t_0|^{-1}=\alpha\ge 1$, the
estimate \eqref{wbound} then follows by comparison to the worst case
sum
$$
\sum_{j=0}^{2^m}(1+\alpha^{-1} j)^{-1}\lesssim
\alpha\bigr(1+|\log(2^m\alpha^{-1})|\,)\,.
$$

To see that equation \eqref{weightedest2} is a consequence of
\eqref{weightedest}, we observe that \eqref{weightedest2} follows by
interpolation from showing, with uniform bounds,
\begin{equation}\label{weightedest4}
  \begin{split}
    \bigl\|m_0(D)\langle\delta^{-2}(x-x_1)\rangle^{-2}S(t_1,t_0)\langle\delta^{-2}(x-x_0)\rangle^{-2}
    f\bigr\|_{L^2(\R)} &\lesssim\|f\|_{L^2(\R)}\,,
    \\
    \bigl\|\langle\delta^{-2}(x-x_1)\rangle^{-2}S(t_1,t_0)\langle\delta^{-2}(x-x_0)\rangle^{-2}m_1(D)
    f\bigr\|_{L^2(\R)} &\lesssim \|f\|_{L^2(\R)}\,.\rule{0pt}{17pt}
  \end{split}
\end{equation}
The second line follows from the first by symmetry and adjoints, so we
prove the estimate of the first line in \eqref{weightedest4}.  We
first note that
$$
\|m_0(D)\langle\delta^{-2}(x-x_1)\rangle^{-2}g\|_{L^2(\R)}\lesssim
\|g\|_{L^2(\R)}+
\|Q_0(t_1)\langle\delta^{-2}(x-x_1)\rangle^{-2}g\|_{L^2(\R)}\,.
$$
The commutator of $Q_0(t_1)$ and
$\langle\delta^{-2}(x-x_1)\rangle^{-2}$ is bounded on $L^2(\R)$; this
uses the fact that $\partial_\xi^2 a_\l(t,x,\xi)\in\l^{-1}S_{\l,\l}$,
that $|t_1-t_0|\le\delta$, and that $\delta^{-3}\le\l$.  Thus,
\eqref{weightedest4} reduces to showing that
$$
\|Q_0(t_1)S(t_1,t_0)\langle\delta^{-2}(x-x_0)\rangle^{-2}f\|_{L^2(\R)}\lesssim
\|f\|_{L^2(\R)}\,,
$$
which follows from \eqref{weightedest}, since the purely spatial
weight $Q_0(t_0)\langle\delta^{-2}(x-x_0)\rangle^{-2}$ is bounded, and
$\delta\le 1$.

To establish \eqref{weightedest}, we calculate
$$
\partial_t\|Q_0(t)S(t,t_0)f\|_{L^2(\R)}^2= 2\,\text{Re}\, \bigl\langle
(\partial_tQ_0+i[Q_0,a_\l^w(t,x,D)])S(t,t_0)f,Q_0(t)S(t,t_0)f\bigr\rangle\,,
$$
so since $|t_1-t_0|\le\delta$ it suffices to show that
$$
\|(\partial_t Q_0+i[Q_0,a_\l^w(t,x,D)])S(t,t_0)f\|_{L^2(\R)} \lesssim
\delta^{-1}\|f\|_{L^2(\R)}+\delta^{-2}\|(x-x_0)f\|_{L^2(\R)}\,.
$$
The operator $\partial_tQ_0+i[Q_0,a_\l^w(t,x,D)]$ is equal to
$$
\delta^{-2}\bigl(\partial_\xi a_\l(t_0,x_0,D)-\partial_\xi
a_\l^w(t,x,D)\bigr) +i\delta^{-2}(t-t_0)\bigl[\partial_\xi
a_\l(t_0,x_0,D),a_\l^w(t,x,D)\bigr]\,.
$$
The commutator term is bounded on $L^2$ since $a_\l\in \l C^1
S_{\l,\l}$, and $\partial_\xi^2 a_\l(t_0,x_0,\xi)\in \l^{-1} C^1
S_{\l,\l}$.  Since $|t-t_0|\le\delta$, the second term is thus bounded
by $\delta^{-1}$.  The $L^2$ norm of the first term is bounded by
$\delta^{-1}+\delta^{-2}(x-x_0)$, so we have to bound
$$
\delta^{-1}\|(x-x_0)S(t,t_0)f\|_{L^2(\R)}\lesssim
\|f\|_{L^2(\R)}+\delta^{-1}\|(x-x_0)f\|_{L^2(\R)}\,.
$$
This follows by Corollary \ref{c:weight}, since
$\delta^{-1}\le\l^{\frac 13}$.  \qed


\newsection{Bilinear $L^2$ estimates}\label{sec:bilinear} We prove
here the bilinear estimate Lemma \ref{bilinearlemma}. For this
section, we will let
$$
\chi_T(t,x)=\bigl(1+\l^{\frac23}|x-x_T(t)|\bigr)^{-N}
$$ 
for some suitably large but fixed $N$, and use the fact that
$|v_T|\lesssim \l^{\frac13}\chi_T$, by Theorem \ref{tubeloc}.  Also in
this section we let
$$
a(t,x,\xi)=a_{\l^{2/3}}(t,x,\xi)\,,\quad a_\xi(t,x,\xi)=\partial_\xi
a_{\l^{2/3}}(t,x,\xi)\,.
$$

We first reduce Lemma \ref{bilinearlemma} to the case that the $c_T$
are constants, that is, to the following lemma. In the case that $a$
is independent of $(t,x)$ the following lemma is a simple consequence
of the proof of the restriction theorem in two dimensions; see for
example \cite[Section IX.5]{Stein}.

\begin{lemma}\label{bilinearlemmab}
  Suppose that $b_T\,,d_S\in\C$. Then, for any subset $\Lambda$ of
  tube pairs $(T,S)$ satisfying $\angle(T,S)\ge\theta$, the following
  bilinear $L^2$ bound holds,
  \begin{equation}\label{l2bib}
    \biggl\| \,\sum_{(T,S)\in\Lambda} b_T v_T\cdot  d_S v_S \,\biggr\|_{L^2} \lesssim
    \theta^{-\frac12}
    \Bigl(\,\sum_T\,|b_T|^2\Bigr)^\frac 12 
    \Bigl(\,\sum_S\,|d_S|^2\Bigr)^\frac 12
    \,.
  \end{equation}
  The norm is taken over the common $\l^{-\frac 13}$ time slice in
  which the tubes lie, and there is no restriction on the number of
  terms.
\end{lemma}

To make the reduction of Lemma \ref{bilinearlemma} to Lemma
\ref{bilinearlemmab}, we note that it suffices by Minkowski to
establish \eqref{l2bi} on an interval $J$ of size $2^{-k}\l^{-\frac
  13}$, including only the $c_{T,j}$ for which $I_{T,j}=J$.  On such
an interval we can write
$$
c_{T,j}(t)=c_{T,j}(t_0)+\int_0^t c'_{T,j}(s)\,ds\,,
$$
where $t_0$ is the left endpoint of $J$.  We can thus bound
\begin{multline*}
  \biggl|\,\sum c_{T,j}v_T\cdot c_{S,k}v_S\biggr|\le \biggl|\,\sum
  c_{T,j}(t_0)v_T\cdot c_{S,k}(t_0)v_S\biggr|+ \int_J \,\biggl|\,\sum
  c'_{T,j}(r)v_T\cdot c_{S,k}(t_0)v_S\biggr|\,dr\\+ \int_J
  \,\biggl|\,\sum c_{T,j}(t_0)v_T\cdot c'_{S,k}(s)v_S\biggr|\,ds+
  \int_{J\times J} \,\biggl|\,\sum c'_{T,j}(r)v_T\cdot
  c'_{S,k}(s)v_S\biggr|\,dr\,ds\,.
\end{multline*}

Bringing the integral out of the $L^2$ norm and applying \eqref{l2bib}
together with the Schwartz inequality yields the desired bound
\begin{equation*}
  \biggl\| \,\sum_{\angle(T,S) \ge \theta} c_{T,j} v_T\cdot  c_{S,k} v_S \,
  \biggr\|_{L^2(J\times\R)}
  \lesssim
  \theta^{-\frac12}\sum_{I_{T,j}=J}\,\Bigl(\, \|c_{T,j}\|_{L^\infty}^2+
  2^{-k}\l^{-\frac13}\|c'_{T,j}\|_{L^2}^2 \,\Bigr)
  \,.
\end{equation*}

We now turn to the proof of Lemma \ref{bilinearlemmab}.  One estimate
we will use is the following. Suppose that the tubes $T$ (respectively
$S$) all point in the same direction, that is, $\xi_T$ is the same for
all $T$, and $\xi_S$ is the same for all $S$, where $|\xi_T-\xi_S|\ge
\l\theta$. Then
\begin{equation}\label{tubeproduct}
  \int \Bigl(\sum_{T,S} |b_T|\chi_T\cdot |d_S|\chi_S\Bigr)^2\,dt\,dx\lesssim 
  \l^{-\frac 43}\theta^{-1}
  \Bigl(\sum_T |b_T|^2\Bigr)\Bigl(\sum_S |d_S|^2\Bigr)\,.
\end{equation}
This follows since different tubes $T$ (respectively $S$) are
disjoint, and the intersection of any pair of tubes $T$ and $S$ is a
$\l^{-\frac23}$ interval in $x$ times a $\theta^{-1}\l^{-\frac23}$
interval in $t$ about the center of the intersection.  Precisely, one
can make a change of variables of Jacobian $\l^{\frac43}\theta$ to
reduce matters to $\l=1$ and $\theta=\tfrac\pi 2$, where the result is
elementary.  We remark that since the terms are positive, this holds
even if the sum over $T$ and $S$ on the left includes just a subset of
the collection of all $T$ and $S$.

Consider now an integral
$$
\int v_T\,v_S\,\overline{v_{T'}}\,\overline{v_{S'}}\,dt\,dx\,.
$$
Recalling that $\l^{-\frac 23}\xi_T\in\Z$, we will relabel
$$
\xi_T=\xi_{m+j}\,,\quad \xi_S=\xi_{m-j}\,,\quad
\xi_{T'}=\xi_{n+i}\,,\quad \xi_{S'}=\xi_{n-i}\,,
$$
corresponding to $\xi_T=\l^{\frac 23}(m+j)$, etc.  Here, $m$, $n$,
$i$, and $j$ take on integer values; for simplicity we assume $m\ne n$
and $i\ne j$, and $i,j$ nonzero. The cases of equality are simpler in
what follows.  By symmetry we may take $j > i \ge 1$.  Assuming that
$\angle (T,S) \ge \theta$ implies $j\ge \l^{\frac 13}\theta$, and
similarly $\angle (T',S') \ge \theta$ implies $i\ge \l^{\frac
  13}\theta$.

We introduce the quantities
\begin{align*}
  v_{(4)}&=v_T\,v_S\,\overline{v_{T'}}\,\overline{v_{S'}}\,,\\
  \rule{0pt}{12pt}
  a_{(4)}&=a(t,x,\xi_{m+j})+a(t,x,\xi_{m-j})-a(t,x,\xi_{n+i})-a(t,x,\xi_{n-i})\,,\\
  \rule{0pt}{12pt}
  a_{(4')}&=a(t,x,\xi_{m+j})+a(t,x,\xi_{m-j})-a(t,x,\xi_{m+i})-a(t,x,\xi_{m-i})\,.
\end{align*}
Since $a_{\xi\xi}\approx\l^{-1}$, we have simultaneous upper and lower
bounds for $a_{(4')}$,
\begin{equation}\label{a4est1}
  a_{(4')}=\l^{\frac 23}\int_{i}^{j}a_\xi(t,x,\xi_{m+s})
  -a_\xi(t,x,\xi_{m-s})
  \;ds
  \approx
  \l^{\frac 13}\int_i^j s\,ds\approx
  \l^{\frac 13}(j^2-i^2)\,.
\end{equation}
We may similarly use $|\partial_{t,x}^{\alpha}a_{\xi\xi}|\lesssim
\l^{-1+\frac 23(|\alpha|-1)}$ for nonzero $\alpha$ to deduce
\begin{equation}\label{a4est2}
  \bigl|\partial_{t,x}^\alpha a_{(4')}\bigr|\lesssim 
  \l^{\frac 13+\frac 23(|\alpha|-1)}(j^2-i^2)\,,\qquad |\alpha|\ge 1\,.
\end{equation}
To control the difference of $a_{(4)}$ and $a_{(4')}$, we introduce
the quantity
\begin{align*}
  r_{(4)}&=a_{(4)}-a_{(4')}-2(\xi_m-\xi_n)a_\xi(t,x,\xi_m)\\
  \rule{0pt}{12pt} &=\sum_{\pm} a(t,x,\xi_{m \pm i})-a(t,x,\xi_{n \pm
    i})-(\xi_{m \pm i}-\xi_{n \pm i})\, a_\xi(t,x,\xi_m)\,,
\end{align*}
which by a Taylor expansion in $\xi$ is seen to satisfy
\begin{equation*}
  \bigl|\partial_{t,x}^\alpha r_{(4)}\bigr|\lesssim 
  \l^{\frac 13+\frac 23\max(0,|\alpha|-1)}(i+|m-n|)\,|m-n|\,.
\end{equation*}

Using a Taylor expansion of $a(t,x,\xi)$ about $\xi=\xi_T$, we can
write
\begin{equation*}
  a^w(t,x,D)v_T=
  \bigl[
  a(t,x,\xi_T)+a_\xi(t,x,\xi_T)(D-\xi_T)+r^w(t,x,D)
  \bigr]v_T\,,
\end{equation*}
where $\l^{-\frac 13}r^w(t,x,D)$ applied to a $\l^{\frac 23}$-scaled
Schwartz function with frequency center $\xi_T$, such as $v_T$, yields
a Schwartz function of comparable norm, uniformly over $\l$.

We then write
\begin{multline*}
  a_\xi(t,x,\xi_T)(D-\xi_T)v_T=
  \bigl[a_\xi(t,x,\xi_T)-a_\xi(t,x,\xi_m)\bigr](D-\xi_T)v_T\\
  +a_\xi(t,x,\xi_m)(D-\xi_T)v_T\,.
\end{multline*}
The same expansions hold with $\xi_T$ replaced by $\xi_S$, $\xi_{T'}$,
and $\xi_{S'}$.

Replacing $\xi_T$ by any of $\xi_{m\pm j}$ or $\xi_{n\pm i}$, the
function $a_\xi(t,x,\xi_T)-a_\xi(t,x,\xi_m)$ satisfies
$$
\bigl|\partial_{t,x}^\alpha\bigl(a_\xi(t,x,\xi_T)-a_\xi(t,x,\xi_m)\bigr)\bigr|\lesssim
(i+j+|m-n|)\l^{-\frac 13+\frac 23\max(0,|\alpha|-1)}\,.
$$

Let $L=D_t-a_\xi(t,x,\xi_m)D$, where $D=D_x$ as always.  Writing $D_t
v_T=a^w(t,x,D)v_T$, and using the above expansion for the latter, we
see that $(L-a_{(4')})v_{(4)}$ can be written as a sum of 5 terms,
\begin{multline*}
  (L-a_{(4')})v_{(4)}=
  (L_Tv_T)\,v_S\,\overline{v_{T'}}\,\overline{v_{S'}}+
  v_T\,(L_Sv_S)\,\overline{v_{T'}}\,\overline{v_{S'}}-
  v_T\,v_S\,\overline{(L_{T'}v_{T'})}\,\overline{v_{S'}}\\-
  v_T\,v_S\,\,\overline{v_{T'}}\overline{(L_{S'}v_{S'})}+
  r_{(4)}v_{(4)}\,,
\end{multline*}
where we wrote $\xi_T+\xi_S-\xi_{T'}-\xi_{S'}=2(\xi_m-\xi_n)$, and
where
\begin{align*}
  L_Tv_T&=\bigl(a^w(t,x,D)-a(t,x,\xi_T)-a_\xi(t,x,\xi_m)(D-\xi_T)\Bigr)v_T\\
  &=\Bigl(\bigl[a_\xi(t,x,\xi_T)-a_\xi(t,x,\xi_m)\bigr](D-\xi_T)+
  r^w(t,x,D)\Bigr)v_T\,.
\end{align*}
In the expressions for $L_S$, $L_{T'}$, and $L_{S'}$, $T$ is
respectively replaced by $S$, $T'$, and $S'$, but the $\xi_m$ is the
same for each. The $L_T$'s thus depend on all 4 subscripts, but this
is fine since the below analysis is applied separately to each term.

Since $(D-\xi_T)$ applied to $v_T$ counts as $\l^{\frac 23}$, then
$|L_T v_T|\lesssim \l^{\frac 13}(i+j+|m-n|)\l^{\frac 13}\chi_T$.
Indeed, $L_T v_T$ can be written, at each fixed time $t$, as
$\l^{\frac 13}(i+j+|m-n|)$ times a Schwartz function of the same scale
and phase space center as $v_T$.

Consequently,
$$
\bigl|(L-a_{(4')})v_{(4)}\bigr|\lesssim \l^{\frac
  13}(i+j+|m-n|)\,\langle m-n\rangle \,\l^{\frac 43}\chi_{(4)}\,,
$$
where $\chi_{(4)}$ is the product of the corresponding $\chi_T$.

We also need the estimate
$$
\bigl|(L-a_{(4')})^2v_{(4)}\bigr|\lesssim \l^{\frac
  23}(i+j+|m-n|)^2\langle m-n\rangle^2 \l^{\frac 43}\chi_{(4)}\,.
$$
Following the above arguments, we can write $(L-a_{(4')})^2v_{(4)}$ as
a sum of 16 terms like
$$
(L_T^2v_T)\,v_S\,\overline{v_{T'}}\,\overline{v_{S'}}+
(L_Tv_T)\,(L_Sv_S)\,\overline{v_{T'}}\,\overline{v_{S'}}-
(L_Tv_T)\,v_S\,\overline{(L_{T'}v_{T'})}\,\overline{v_{S'}}- \cdots
$$
plus 4 commutator terms
$$
\bigl(\,\bigl[D_t-a^w(t,x,D)\,,L_T\bigr]v_T\bigr)\,v_S\,
\overline{v_{T'}}\,\overline{v_{S'}}+
v_T\,\bigl(\,\bigl[D_t-a^w(t,x,D)\,,L_S\bigr]v_S\bigr)\,
\overline{v_{T'}}\,\overline{v_{S'}}-\cdots
$$
plus remainder terms
$$
(Lr_{(4)})\,v_{(4)}+r_{(4)}(L-a_{(4')})v_{(4)}\,.
$$
Except for the commutator terms, the desired bounds follow by the same
estimates as for $(L-a_{(4')})v_{(4)}$.  The commutator terms depend
on the fact that the commutator of two symbols in $C^1
S_{\l,\l^{2/3}}$ is in $\l^{-1}S_{\l,\l^{2/3}}$.

We let $\phi$ be a smooth cutoff to a $\l^{-\frac 13}$ time interval
in $t$. We can then write
\begin{multline*}
  \int v_{(4)}\,\phi\,dt\,dx= \int
  \frac{\bigl(L-a_{(4')}\bigr)^2v_{(4)}}{a_{(4')}^2}\,\phi\,dt\,dx
  +\int \frac{\bigl(\phi (L'a_{(4')})-a_{(4')}(L'\phi)\bigr)\,v_{(4)}}{a_{(4')}^2}\,dt\,dx\\
  -\int\frac{\bigl(2\phi
    (L'a_{(4')})-a_{(4')}(L'\phi)\bigr)\bigl(L-a_{(4')}\bigr)\,v_{(4)}}{a_{(4')}^3}
  \,dt\,dx\,,
\end{multline*}
where we have integrated by parts in the last two terms, and $L'$ is
the transpose of $L$.  We assume here that the $v_T$ are extended to a
slightly larger time interval to allow integration by parts in $t$.

By the above estimates the integrand of the first term on the right is
bounded by
$$
\frac{(i+j+|m-n|)^2\langle m-n\rangle^2 \l^{\frac
    43}\chi_{(4)}}{(j^2-i^2)^2} \le \frac{\langle
  m-n\rangle^4}{\langle \,j-i\,\rangle^2}\,\l^{\frac 43}\chi_{(4)}\,.
$$
By \eqref{a4est1} and \eqref{a4est2}, we have $(L'a_{(4')})\lesssim
a_{(4')}$.  The integrands of the last two terms on the right hand
side are then respectively dominated by
$$
\frac{\l^{\frac 43}\chi_{(4)}}{(j^2-i^2)}+ \frac{(i+j+|m-n|)\langle
  m-n\rangle\l^{\frac 43}\chi_{(4)}}{(j^2-i^2)^2} \le \frac{\langle
  m-n\rangle^2}{\langle j-i\rangle^2}\,\l^{\frac 43}\chi_{(4)}\,.
$$ 
Consequently, we have shown that we can write $\int
v_{(4)}\phi\,dt\,dx=\int w_{(4)}\,dt\,dx$, where
$$
|w_{(4)}|\lesssim \frac{\langle m-n\rangle^4}{\langle
  j-i\rangle^2}\,\l^{\frac 43}\chi_{(4)}\,.
$$
Given a collection $\Lambda$ of pairs of tubes $(T,S)$, we let
$$
b_n=\Bigl(\sum_{\xi_T=\xi_n}|b_T|^2\Bigr)^{\frac 12}\,,\qquad
d_n=\Bigl(\sum_{\xi_S=\xi_n}|d_S|^2\Bigr)^{\frac 12}\,,
$$
where the sum is over all $T$, respectively $S$, in the collection
that have frequency center $\xi_n$. Then
\begin{align*}
  \Bigl\|\,\phi\!\!\!\!\sum_{(T,S)\in\Lambda} b_T v_T\cdot d_S
  v_S\Bigr\|_{L^2}^2&=
  \sum_{(T,S)\in\Lambda}\sum_{(T',S')\in\Lambda}b_T\,d_S\,\overline{b_{T'}}\,\overline{d_{S'}}
  \int v_{(4)}\phi\,dt\,dx\\
  &\le\; \sum_{m,j}\sum_{n,i}
  \sum_{(T,S)\in\Lambda_{m,j}}\sum_{(T',S')\in\Lambda_{n,i}}
  \bigl|\,b_T\,d_S\,\overline{b_{T'}}\,\overline{d_{S'}}\,\bigr| \int
  w_{(4)}\,dt\,dx\,,
\end{align*}
where $\Lambda_{m,j}\subseteq\Lambda$ consists of the pairs
$(T,S)\in\Lambda$ such that $\xi_T=\xi_{m+j}$, and $\xi_S=\xi_{m-j}$.

By the above this is dominated by
\begin{multline*}
  \l^{\frac 43}\sum_{m,n,i,j}\frac{\langle
    m-n\rangle^4}{\langle\,i-j\,\rangle^2}\int
  \Bigl(\sum_{(T,S)\in\Lambda_{m,j}}|\,b_T\,d_S|\,\chi_T\chi_S \Bigr)
  \Bigl(\sum_{(T',S')\in\Lambda_{n,i}}|\,b_{T'}\,d_{S'}|\,\chi_{T'}\chi_{S'}\Bigr)
  \,dt\,dx\\
  \lesssim \theta^{-1}\sum_{m,n,i,j}\frac{\langle
    m-n\rangle^2}{\langle i-j\rangle^2}
  b_{m+j}\,d_{m-j}\,b_{n+i}\,d_{n-i}\,,
\end{multline*}
where we used the Cauchy-Schwartz inequality and \eqref{tubeproduct}.

We next show that we may write $\int v_{(4)}\phi\,dt\,dx=\int
w_{(4)}\,dt\,dx$, where
\begin{equation}\label{mnbound}
  |w_{(4)}|\lesssim 
  \frac{1}{\langle m-n\rangle^{18}}\,\l^{\frac 43}\chi_{(4)}\,.
\end{equation}
Since
$$
\min\Bigl(\frac{\langle m-n\rangle^4}{\langle\,i-j\,\rangle^2}\,,
\,\frac 1{\langle m-n\rangle^{18}}\Bigr)\le \frac{1}{\langle\,
  i-j\,\rangle^{\frac 32}\langle m-n\rangle^{\frac 32}}\,,
$$
this will establish that
$$
\Bigl\|\,\phi\!\!\!\!\sum_{(T,S)\in\Lambda} \, b_T v_T\cdot d_S
v_S\Bigr\|_{L^2}^2\lesssim
\theta^{-1}\sum_{m,n,i,j}\frac{1}{\langle\,i-j\,\rangle^{\frac
    32}\langle m-n\rangle^{\frac 32}}
b_{m+j}\,d_{m-j}\,b_{n+i}\,d_{n-i}\,.
$$
By Schur's lemma, this is in turn bounded by
\begin{equation*}
  \theta^{-1}\Bigl(\sum_{m,j} b_{m+j}^2\,d_{m-j}^2\Bigr)^{\frac 12}
  \Bigl(\sum_{n,i}b_{n+i}^2\,d_{n-i}^2\Bigr)^{\frac 12}
  \le\theta^{-1}\Bigl(\sum_T |b_T|^2\Bigr)\Bigl(\sum_S |d_S|^2\Bigr)\,,
\end{equation*}
where the last sum is over all $T,S$ that occur in $\Lambda$.

To prove \eqref{mnbound}, we write
$$
2(m-n)\!\int v_{(4)}\phi\,dt\,dx= \int\l^{-\frac
  23}(\xi_T+\xi_S-\xi_{T'}-\xi_{S'})\,v_{(4)}\phi\,dt\,dx=\int
w_{(4)}\phi\,dt\,dx \,,
$$
where
$$
w_{(4)}=(\l^{-\frac
  23}(D-\xi_{T})\,v_T)\,v_S\,\overline{v_{T'}}\,\overline{v_{S'}}+
v_T\,(\l^{-\frac
  23}(D-\xi_{S})\,v_S)\,\overline{v_{T'}}\,\overline{v_{S'}}-\cdots
$$
We repeat this process, and use that $|\l^{-\frac{2k}3}(D-\xi_T)^k
v_T|\lesssim\l^{\frac 13}\chi_T$.  \qed


\newsection{Results for dimension $d\ge 3$}\label{sec:highdimension}
In this section we work on a compact $d$-dimensional manifold $M$
without boundary. We consider spectral clusters for $\g,\rho\in
\text{Lip(M)}$ exactly as in Theorem \ref{theorem1.1}.  We will apply
the general procedure of the previous sections to prove the following,
which establishes the conjectured result \eqref{lipconj} for a partial
range of $p$.  The restriction on $p$ is partly due to the below
Propositions \ref{p:bushcount'} and \ref{p:shorttime'} being weaker
than one would hope for. In particular, Proposition \ref{p:shorttime'}
uses only Strichartz estimates.  It is not clear what the analogue of
the bilinear estimates used for $d=2$ to handle large angle
interactions should be in this case. The bound of Proposition
\ref{p:bushcount} also is strictly larger when $2^ma^2\ll 1$ than the
bound suggested by heuristic arguments.

\begin{theorem}\label{t:hid}
  Let $u$ be a spectral cluster on $M$, where $M$ is of dimension
  $d\ge 3$. Then
  \begin{equation}\label{qboundf}
    \|u\|_{L^{p}(M)}\le C_p\, \lambda^{d(\frac 12-\frac 1p)-\frac 12}\|u\|_{L^{2}(M)}\,,\quad 
    \tfrac{6d-2}{d-1}<p\le\infty\,.
  \end{equation}
\end{theorem}
The following partial range result for the other estimate in
\eqref{lipconj},
\begin{equation*}
  \|u\|_{L^p(M)}\le 
  C\,\l^{\frac{2(d-1)}{3}(\frac 12-\frac 1p)}\,\|f\|_{L^2(M)}\,,
  \qquad 2\le p\le \tfrac{2(d+1)}{d-1}\,,
\end{equation*}
was established in \cite{Sm}, as was the $p=\infty$ case of
\eqref{qboundf}.

The proof of Theorem \ref{t:hid} follows the same general steps as
Theorem \ref{theorem1.1}, and so we focus below on the modifications
necessary in each step.

{\bf Step 1:} {\em Reduction to a frequency localized first order
  problem.}  Care must be taken in the frequency localization step to
handle the high-frequency terms, since Sobolev embedding as used in
the $d=2$ case is not sufficient to establish the desired result for
large $p$ in high dimensions. In particular, the analogue of Theorem
\ref{l8theorem} does not hold for $p=\tfrac{6d-2}{d-1}$. Instead, we
use the following estimate from \cite{Sm}, valid for Lipschitz
$\g,\rho$,
\begin{equation}\label{linfest}
  \|u\|_{L^\infty(M)}\lesssim \l^{\frac {d-1}2 }\|u\|_{L^2(M)}\,.
\end{equation}
We remark that this estimate used the strict spectral localization of
$u$ and intrinsic Sobolev embedding on $M$ to deduce it from results
for smaller $p$.

By \eqref{linfest}, if $\phi$ is a bump function supported in a local
coordinate patch, and
$$
\phi u=(\phi u)_{<\l}+(\phi u)_{\l}+(\phi u)_{>\l}
$$
is the decomposition of $\phi u$ into terms with local-coordinate
frequencies respectively less than $c\l$, comparable to $\l$, and
greater than $c^{-1}\l$, then each term in the decomposition has
$L^\infty$ norm bounded by $\l^{\frac {d-1}2}\|u\|_{L^2(M)}\,.$ The
proof of \cite[Corollary 5]{Sm}, together with \eqref{ellipticest} and
Sobolev embedding on $\R^d$, yields
$$
\|(\phi u)_{<\l}\|_{L^{\frac {2d}{d-2}}}+\|(\phi u)_{>\l}\|_{L^{\frac
    {2d}{d-2}}}\lesssim \|u\|_{L^2(M)}\,.
$$
Interpolation with \eqref{linfest} then yields even better bounds than
those of Theorem \ref{t:hid} for these terms.

Hence we are reduced to bounding $\|(\phi u)_\l\|_{L^p}$. With
$a_\l^w(t,x,D)$ and $S(t,t_0)$ defined as they are for $d=2$, where
$x$ and $\xi$ are now of dimension $d-1$, we then reduce matters as
before to establishing
\begin{equation*}
  \|u\|_{L^p([0,1]\times\R^{d-1})} \le 
  C_p \,\lambda^{d(\frac 12-\frac 1p)-\frac 12}\|u_0\|_{L^2(\R^{d-1})}\,,\quad
  u=S(t,0)u_0\,,\quad \tfrac{6d-2}{d-1}<p\le\infty\,,
\end{equation*}
with $\widehat u_0$ supported in $|\xi|\le \frac 34\l$.  As before we
will take $\|u_0\|_{L^2(\R^{d-1})}=1$.

The expansion of $u$ in terms of tube solutions $v_T$ on each
$\l^{-\frac 13}$ time slab and the definition of $A_{a,k,m}$ bushes
then proceeds for $d\ge 3$ the same as for $d=2$, but where we take
$\epsilon=\l^{-\frac{d-1}3}$ as the lower bound for $a$ in the sum
$u=\sum u_{a,k}$ in order to trivially obtain the desired bounds for
$u_\epsilon$.

In $d$-dimensions, a $2^m$-bush has angular spread at least $2^{\frac
  m{d-1}}\l^{-\frac 13}$, and so can retain full overlap for time
$\delta t=2^{-\frac m{d-1}}\l^{-\frac 13}$.  Thus, in dimension $d$ we
decompose the unit time interval into a collection $\I_m$ of intervals
of size $\delta t = 2^{-\lfloor \frac m{d-1}\rfloor}\l^{-\frac 13}$;
such intervals are then dyadic subintervals of the decomposition into
$\l^{-\frac 13}$ time slabs.

The proof of Theorem \ref{t:hid} will then be concluded using the
following two propositions.

\begin{proposition}
  There are at most $\lambda^{\frac13} 2^{-\frac{m}{d-1}} (2^m
  a^{2})^{-2} \langle\log (2^m a^2)\rangle$ intervals $I\in\I_m$ which
  intersect $ A_{a,k,m}$.
  \label{p:bushcount'}\end{proposition}

\begin{proposition}
  For each interval $I\in\I_m$, we have
  \begin{equation*}
    \| u_{a,k}\|_{L^p( A_{a,k,m} \cap I\times\R^{d-1})}^p \lesssim 
    \lambda \bigl(\lambda^{\frac{d-1}3}2^m a\bigr)^{p-p_d}2^{-\frac{kp_d}2}\,,
    \qquad p\ge p_d=\tfrac{2(d+1)}{d-1}\,.
  \end{equation*}
  \label{p:shorttime'}\end{proposition}

Indeed, combining the two propositions we obtain
\[
\begin{split}
  \| u_{a,k}\|_{L^p( A_{a,k,m})}^p \lesssim & \ \lambda
  \bigl(\lambda^{\frac{d-1}3} 2^m a)^{p-p_d} \lambda^{\frac13}
  2^{-\frac{m}{d-1}} (2^m a^{2})^{-2} \langle\log (2^m a^2)\rangle
  2^{-\frac{kp_d}2} \\ \lesssim & \ \lambda^{\frac43 +
    \frac{d-1}{3}(p-p_d)} 2^{m(\frac{p-p_d}2-\frac 1{d-1})} (2^m
  a^2)^{\frac{p-p_d}2-2}\langle\log (2^m a^2)\rangle
  2^{-\frac{kp_d}2}\,.
\end{split}
\]
Recall that $2^m$ and $a$ both take on dyadic values such that
$$
2^m a^2\lesssim 1\,,\qquad 2^m\lesssim \l^{\frac {d-1}3}\,.
$$
When the exponent of $2^m a^2$ is positive,
\[
p > p_d + 4 = \tfrac{6d-2}{d-1}\,,
\]
we may sum over $m$, $a$ and $k$ to obtain
\[
\| u\|_{L^p}^p \le C_p^p\,\lambda^{1+ \frac{d-1}{2}(p-p_d)}\,,
\]
giving the desired result
\[
\| u\|_{L^p} \le C_p\,\lambda^{\frac{d-1}{2}- \frac{d}p}\,.
\]
By \eqref{linfest} the constant $C_p$ remains bounded as $p\rightarrow
\infty$, but may diverge as $p\rightarrow \tfrac{6d-2}{d-1}$. On the
other hand $C_p$ is bounded by a power of $\log\l$ for
$p=\tfrac{6d-2}{d-1}$, since there are only $\approx\log\l$ terms in
each index.


\bigskip

\noindent{\bf Proof of Proposition \ref{p:bushcount'}.}
There are $\approx 2^\frac m{d-1} \l^\frac13$ intervals in $\I_m$, so
we may assume that $a^2 \gg 2^{-m(1+\frac 1{d-1})}$.  It suffices to
prove, for $\epsilon$ a fixed small number, that if among $\epsilon
\,2^{m(1+\frac 2{d-1})} a^2$ consecutive slices in $\I_m$ there are
$M$ slices that intersect $A_{a,k,m}$, then
\begin{equation}\label{Mbound'}
  M\lesssim (2^m a^2)^{-1} \langle \log (2^m a^2)\rangle\,.
\end{equation}

Consider a collection $\{B_n\}_{n=1}^M$ of $M$ distinct
$(a,k,m)$-bushes, centered at $(t_n,x_n)$, such that
\[
\epsilon\,\l^{-\frac 13}2^{m(1+\frac 1{d-1})}a^2\ge |t_n-t_{n'}|\ge
\l^{-\frac 13}2^{-\frac m{d-1}} \quad\text{when}\quad n\ne n'\,.
\]
Denote by $\{v_{n,l}\}_{l=1,2^m}$ the collection of $2^m$ packets in
$B_n$.  As in the proof of Proposition \ref{p:bushcount}, for each $n$
we define the bounded projection operators $P_n$ on $L^2(\R^{d-1})$ at
time $t_n$ by
\[
P_n f = 2^{-m} a^{2} \left(\sum_{l}\overline{c_{n,l}}(t_n)^{-1}
  v_{n,l}(t_n,\cdot\,)\right) \left( \sum_l c_{n,l}(t_n)^{-1}\langle
  v_{n,l}(t_n, \cdot\,), f\rangle\right)
\]
where we recall that $|c_{n,l}(t_n)| \approx a$, so that
\begin{equation}
  \| P_n u(t_n,\cdot\,)\|_{L^2(\R^{d-1})}^2 \approx 2^m a^2\,.
  \label{ag'}\end{equation}

As with Proposition \ref{p:bushcount}, the key estimate is the
following analogue of Lemma \ref{l:packet}.  We remark that the
heuristics of tracking bicharacteristics to count tube-solution
overlaps would suggest that \eqref{ao'} below should hold with bound
$2^{-m}\alpha^{d-1}$ on the right hand side, which would improve the
bound in Proposition \ref{p:bushcount'} to $\l^{\frac 13} 2^{-\frac
  m{d-1}} (2^m a^2)^{-(1+\frac 1{d-1})}\,.$ The fact that the weight
$Q_0$ in \eqref{weightedest} only gives an order one localization of
the energy, however, restricts us to the bound below.

\begin{lemma}
  Let $\alpha=\max\bigl(\l^{-\frac 13}|t_{n'}-t_n|^{-1},\l^{\frac
    13}|t_{n'}-t_n|\bigr)$.  Then the operators $P_n$ satisfy
  \begin{equation}
    \| P_{n'} S(t_{n'},t_n) P_n \|_{L^2(\R^{d-1}) \to L^2(\R^{d-1})} \lesssim 
    2^{-\frac m{d-1}} \alpha \,.
    \label{ao'}
  \end{equation}
  \label{l:packet'}
\end{lemma}
\begin{proof}
  We follow the proof of Lemma \ref{l:packet} at the end of Section
  \ref{sec:wavepacket}.  The same steps follow, where the $q_j$ are
  vector valued if $d\ge 3$.  The analogue of \eqref{wbound} to be
  proven is
$$
\bigl\|m_0(D)^{-\frac
  12}\langle\delta^{-2}(x-x_1)\rangle^{2}w_1\|_{L^2(\R)}^2\lesssim
2^{m(1-\frac 1{d-1})}\alpha\,,
$$
which is established by comparison to the worst case sum, where $j\in
\Z^{d-1}$,
$$
\sum_{|j|\le 2^{\frac m{d-1}}}(1+\alpha^{-1} |j|)^{-1}\lesssim
\alpha\int_0^{2^{\frac m{d-1}}}r^{d-3}\,dr \lesssim 2^{m(1-\frac
  1{d-1})}\alpha\,,
$$
where we used that $\alpha\ge 1$ to handle the $j=0$ term.
\end{proof}

As before, Lemma \ref{l:packet'} leads to the bound
\begin{equation}
  \sum_n \| P_n u\|_{L^2(\R^{d-1})} \lesssim   
  C^{\frac12}
  \| u_0\|_{L^2(\R)}\,, \qquad C =  M+ \sum_{n,n'}
  { 2^{-\frac m{d-1}}\alpha }\,.
  \label{ug'}\end{equation}

Comparing \eqref{ag'} and \eqref{ug'} applied to $u$, it follows that
\begin{equation}\label{Msumbound'}
  2^m a^2 M^2 \lesssim  M+ \sum_{n\ne n'}{ 2^{-\frac m{d-1}}\alpha }
  \,.
\end{equation}

The bound \eqref{Mbound'} is trivial if $2^m a^2 M^2 \lesssim M$, so
we consider the summation term.  For the sum over terms where
$\l^{\frac 13}|t_{n'}-t_n|\ge 1$ we have
$$
\sum_{n\ne n'}2^{-\frac m{d-1}}\l^{\frac 13}|t_{n'}-t_n|\, \lesssim
\epsilon\, 2^m a^2 M^2\,.
$$
Taking $\epsilon$ small we can thus absorb these terms into the left
hand side of \eqref{Msumbound'}.

To conclude, we may assume then that
\[
2^m a^2 M^2 \lesssim 2^{-\frac{m}{d-1}}\l^{-\frac 13}\sum_{n\ne
  n'}|t_{n'}-t_n|^{-1}\,.
\]
We use the $2^{-\frac{m}{d-1}}\l^{-\frac 13}$ separation of the
$t_n$'s to bound
\[
\sum_{n\ne n'}|t_{n'}-t_n|^{-1} \lesssim 2^{\frac{m}{d-1}}\l^{\frac
  13} M \log M\,,
\]
thus
\[
2^m a^2 M \lesssim \log M\,,
\]
or
\begin{equation*}
  M \lesssim (2^m a^2)^{-1} \langle\log (2^m a^2)\rangle\,.
\end{equation*}
\qed


\bigskip

\noindent{\bf Proof of Proposition \ref{p:shorttime'}.}
We estimate $\|u_{a,k}\|_{L^p( A_{a,k,m})}$ on a single slice
$I\times\R^{d-1}$, where $I\in \I_m$. By \eqref{Nakbound}, if there
are $N$ terms in the sum for $u_{a,k}$, then $Na^2\lesssim 2^{-k}$, so
by orthogonality the total energy of $u_{a,k}$ is $\lesssim 2^{-\frac
  k2}$. We then have the Strichartz estimates
\[
\| u_{a,k}\|_{L^{p_d}(I\times\R^{d-1})} \lesssim \lambda^{\frac
  1{p_d}}2^{-\frac k2}\,,\qquad p_d=\tfrac{2(d+1)}{d-1}\,.
\]
If $k>\frac m{d-1}$, this is proven on each $2^{-k}\l^{-\frac 13}$
dyadic subinterval of $I$ then summed.

We interpolate this with the $L^\infty$ bound
\[
\| u_{a,k}\|_{L^\infty(A_{a,k,m}\cap I\times\R^{d-1})} \lesssim
\lambda^{\frac{d-1}3} 2^m a\,,
\]
to obtain
\[
\|u_{a,k}\|_{L^p(A_{a,k,m}\cap I\times\R^{d-1})}^p \lesssim \lambda
(\lambda^{\frac{d-1}3} 2^m a)^{p-p_d}2^{-\frac {kp_d}2}\,.\qed
\]



\begin{thebibliography}{10}

  \bibliographystyle{plain}

\bibitem{BSS} M. Blair, H. Smith, and C. Sogge, \newblock{\em On
    Strichartz estimates for Schr\"odinger operators in compact
    manifolds with boundary}, \newblock Proc. Amer. Math. Soc., {\bf
    136} (2008), 247--256.

\bibitem{CM} R. Coifman and Y. Meyer, \newblock{Commutateurs
    d'integrales singulieres et operateurs multilineaires}, \newblock
  Ann.~Inst.~Fourier~Grenoble {\bf 28} (1978), 177--202.

\bibitem{Gr} D. Grieser, \newblock{\em $L^p$ bounds for eigenfunctions
    and spectral projections of the Laplacian near concave
    boundaries}, Thesis, UCLA, 1992.

\bibitem{weyl} L. H\"ormander, \newblock{\em The Weyl calculus of
    pseudo-differential operators}, \newblock Comm. Pure
  Appl. Math. {\bf 32} (1979), 359--443.


\bibitem{KST1} H. Koch, H. Smith, and D. Tataru, \newblock{\em Sharp
    $L^q$ bounds on spectral clusters for Holder metrics}, \newblock
  Math. Res. Lett. {\bf 14} (2007), 77--85.

\bibitem{KST2} \bysame, \newblock{\em Subcritical $L^p$ bounds on
    spectral clusters for Lipschitz metrics}, \newblock
  Math. Res. Lett. {\bf 15} (2008), 993--1002.

\bibitem{KT} H. Koch and D. Tataru, \newblock{\em Dispersive estimates
    for principally normal operators}, \newblock Comm. Pure
  Appl. Math. {\bf 58} (2005), 217--284.

\bibitem{KTZ} H. Koch, D. Tataru and M. Zworski, \newblock{\em
    Semiclassical $L^p$ estimates}, Ann. Henri Poincar\'e {\bf 8}
  (2007), no. 5, 885--916.

\bibitem{Sm1} H. Smith, \newblock{\em Spectral cluster estimates for
    $C^{1,1}$ metrics}, \newblock Amer. Jour. Math. {\bf 128} (2006),
  1069--1103.

\bibitem{Sm} \bysame, \newblock{\em Sharp $L^2 \rightarrow L^q$ bounds
    on spectral projectors for low regularity metrics}, \newblock
  Math. Res. Lett. {\bf 13} (2006), 965--972.

\bibitem{SmSo} H. Smith and C. Sogge, \newblock{\em On Strichartz and
    eigenfunction estimates for low regularity metrics}, \newblock
  Math. Res. Lett. {\bf 1} (1994), 729--737.

\bibitem{SmSo2} \bysame, \newblock{\em On the $L^p$ norm of spectral
    clusters for compact manifolds with boundary}, \newblock Acta
  Math. {\bf 198} (2007), 107--153.

\bibitem{SmTa} H. Smith and D. Tataru, \newblock{\em Sharp
    counterexamples for Strichartz estimates for low regularity
    metrics}, \newblock Math. Res. Lett. {\bf 9} (2002), 199--204.

\bibitem{So} C. Sogge, \newblock{\em Concerning the $L^p$ norm of
    spectral clusters for second order elliptic operators on compact
    manifolds}, \newblock J. Funct. Anal. {\bf 77} (1988), 123--134.

\bibitem{Stein} E.~M.~Stein, \newblock{Harmonic Analysis: Real
    Variable Methods, Orthogonality, and Oscillatory
    Integrals. Princeton University Press, 1993.}

\bibitem{Tat} D. Tataru, \newblock{\em Strichartz estimates for
    operators with nonsmooth coefficients III}, \newblock
  J. Amer. Math. Soc. {\bf 15} (2002), 419--442.

\bibitem{Tay} M. Taylor, \newblock{Pseudodifferential Operators and
    Nonlinear PDE}, Progress in Mathematics, vol. 100, Birkh\"auser,
  Boston, 1991.

\end{thebibliography}
\end{document}